\documentclass[a4paper,11pt]{article}
\usepackage{graphicx}        
\usepackage{multicol}        
\usepackage[bottom]{footmisc}
\usepackage{amsmath}

\usepackage{epsfig}
\usepackage{authblk}
\usepackage{graphicx}
\usepackage{float}
\usepackage{diagbox}
\usepackage{nicefrac}
\usepackage[T1]{fontenc}
\usepackage{blkarray}
\usepackage{mathtools,amsthm,amsfonts}
\usepackage{helvet}  

\usepackage{newtxtext}
\usepackage[upint,varg]{newtxmath}

\usepackage{hhline}
\usepackage{tikz}

\usepackage{color}

\usepackage{bm}
\usepackage{nicefrac}
\usepackage[hmargin={28mm,28mm},vmargin={30mm,35mm}]{geometry}
\usepackage{array}
\usepackage{pdfpages}
\usepackage{subfig}
\usepackage{enumitem}
\usepackage{bbold}
\usepackage{xparse}								
\usepackage{stmaryrd}							
\usepackage{sidecap}								
\usepackage{upgreek}
\usepackage{accents}
\allowdisplaybreaks
\usepackage{tcolorbox}
\usepackage{algorithmic}
\usepackage[ocgcolorlinks,allcolors={blue}]{hyperref}

\usetikzlibrary{positioning}
\usepackage{alltt}

\usepackage[bbgreekl]{mathbbol}

\DeclareSymbolFontAlphabet{\mathbb}{AMSb}
\DeclareSymbolFontAlphabet{\mathbbl}{bbold}

\def\N{\mathbb N}
\def\R{\mathbb R}
\def\Po{\mathbb P}

\def\U{\mathbf{U}}

\def\D{\mathcal{D}}

\def\T{{\mathbf{T}}}

\def\<{\langle}
\def\>{\rangle}

\def\lsim{\lesssim}

\def\wt{\widetilde}
\def\Chi{\raise .3ex \hbox{\large $\chi$}} 

\def\ov{\overline}

\def\normal{\mathbf{n}}
\def\n{\normal}
\def\tang{{\boldsymbol{\tau}}}
\def\trescacof{\text{g}}
\def\m{{\boldsymbol{\mu}}}
\def\la{{\boldsymbol{\lambda}}}

\newcommand{\dist}{\mathop{dist}}

\newcommand{\ProjFG}{\pi^0_{\faces_\Gamma}}
\newcommand{\ProjM}{\pi^0_{\cells}}

\newcommand{\HGbracket}[2]{\langle #1,#2\rangle_{\Gamma}}
\newcommand{\IPbracket}[2]{\langle #1,#2\rangle}

\def\[{\Bigl [}
\def\]{\Bigr ]}
\def\({\Bigl (}
\def\){\Bigr )}

\def\dsp{\displaystyle}
\def\x{{\bf x}}
\def\y{{\bf y}}

\def\U{{\bf U}}

\def\q{{\bf q}}

\def\cells{{\mathcal{M}}}
\def\faces{{\mathcal{F}}}
\def\nodes{{\mathcal{V}}}
\def\edges{{\mathcal{E}}}

\def\x{{\bf x}}
\def\dsp{{\displaystyle x}}
\def\d{{\rm d}}
\def\div{{\rm div}}

\def\dsp{\displaystyle}
\def\bu{\mathbf{u}}
\def\bq{\mathbf{q}}
\def\bv{\mathbf{v}}
\def\bw{\mathbf{w}}

\newcommand{\bs}[1]{\boldsymbol{#1}}

\newcommand{\UD}{\mathbf{U}_\D}
\newcommand{\UDz}{\mathbf{U}_{0,\D}}
\newcommand{\MD}{\mathbf{M}_\D}
\newcommand{\CD}{\mathbf{C}_\D}
\newcommand{\Poly}[1]{\mathbb{P}^{#1}}
\newcommand{\JS}{H^{\nicefrac 12}_{0,{\rm j}}(\Gamma)}    
\newcommand{\JSdual}{H^{-\nicefrac 12}_{0,{\rm j}}(\Gamma)}
\newcommand{\Hhalf}{H^{\nicefrac 12}_0(\Gamma)}    
\newcommand{\Hhalfdual}{H^{-\nicefrac 12}_0(\Gamma)}    

\newcommand{\IUD}{\mathcal I_{\UDz}} 
\newcommand{\IMD}{\mathcal I_{\MD}} 
\newcommand{\ave}{\mathrm{a}}

\newcommand{\NORM}[2]{\|#2\|_{#1}}
\newcommand{\SEMINORM}[2]{|#2|_{#1}}

\def\O{\Omega}

\newcommand{\Ks}{{\mathcal{K}s}}
\def\Ksig{{K\!\sigma}}
\def\Ls{{\mathcal{L}s}}
\def\Lsig{{L\!\sigma}}
\newcommand{\sige}{{\sigma e}}
\newcommand{\Ksi}{{\Ks(i)}}
\newcommand{\Lsi}{{\mathcal{L}s(i)}}

\def\bbsig{\bbsigma}
\def\bbeps{\bbespilon}

\newcommand{\jump}[1]{\llbracket #1 \rrbracket}

\newcommand{\email}[1]{\href{mailto:#1}{#1}}

\begingroup
    \makeatletter
    \@for\theoremstyle:=definition,remark,plain\do{%
        \expandafter\g@addto@macro\csname th@\theoremstyle\endcsname{%
            \addtolength\thm@preskip\parskip
            }%
        }
\endgroup

\newtheorem{theorem}{Theorem}
\numberwithin{theorem}{section}
\newtheorem{proposition}[theorem]{Proposition}
\newtheorem{lemma}[theorem]{Lemma}

\theoremstyle{remark}
\newtheorem{remark}[theorem]{Remark}
\theoremstyle{definition}

\newtheorem{definition}[theorem]{Definition}

\def\trace{\gamma}
\def\grad{\nabla}

\def\div{{\rm div}}
\def\dist{\mbox{\rm dist}}

\definecolor{labelkey}{rgb}{0.6,0,1}

\usepackage{parskip}
\usepackage{algorithmic}

\usepackage{diagbox}
\usepackage{booktabs}

\makeatletter
\def\thm@space@setup{%
  \thm@preskip=\parskip \thm@postskip=0pt
}
\makeatother

\begin{document}
\title{Analysis of a VEM--fully discrete polytopal scheme with bubble stabilisation for contact mechanics with Tresca friction}

\author[2]{{J\'er\^ome Droniou}\footnote{\email{jerome.droniou@umontpellier.fr}}}
\author[1,3]{{Ali Haidar}\footnote{Corresponding author, \email{ali.haidar@univ-cotedazur.fr}}}
\author[1]{{Roland Masson}\footnote{Corresponding author, \email{roland.masson@univ-cotedazur.fr}}}
\affil[1]{Universit\'e C\^ote d'Azur, Inria, Laboratoire J.A. Dieudonn\'e, Nice, France}%
\affil[2]{IMAG, Univ. Montpellier, CNRS, Montpellier, France. School of Mathematics, Monash University, Australia}%
\affil[3]{IFP Energies nouvelles, department of mathematics, 92500 Rueil-Malmaison, France}%
\date{}
\maketitle
\vspace{-.5cm}
{\footnotesize{
\centering
}}

\begin{abstract}
This work performs the convergence analysis of the polytopal nodal discretisation of contact-mechanics (with Tresca friction) recently introduced in \cite{droniou2023bubble} in the framework of poro-elastic models in fractured porous media. The scheme is based on a mixed formulation, using face-wise constant approximations of the Lagrange multipliers along the fracture network and a fully discrete first order nodal approximation of the displacement field. The displacement field is enriched with additional bubble degrees of freedom along the fractures to ensure the inf–sup stability with the Lagrange multiplier space. It is presented in a fully discrete formulation, which makes its study more straightforward, but also has a Virtual Element interpretation.

The analysis establishes an abstract error estimate accounting for the fully discrete framework and the non-conformity of the discretisation. A first order error estimate is deduced for sufficiently smooth solutions both for the gradient of the displacement field and the Lagrange multiplier.  A key difficulty of the numerical analysis is the proof of a discrete inf-sup condition, which is based on a non-standard $H^{-\nicefrac12}$-norm (to deal with fracture networks) and involves the jump of the displacements, not their traces. The analysis  also requires the proof of a discrete Korn inequality for the discrete displacement field which takes into account fracture networks.
Numerical experiments based on analytical solutions confirm our theoretical findings. 
  
\end{abstract}

%

\section{Introduction}

The simulation of poromechanical models in fractured (or faulted) porous rocks plays an important role in many subsurface applications  such as the assessment of fault reactivation risks in CO2 storage or the hydraulic fracture stimulation in deep geothermal systems. These models couple the flow along the fractures and the surrounding matrix to the rock mechanical deformation and the mechanical behavior of the fractures. Fractures are classically represented as a network of planar surfaces connected to the surrounding matrix domain, leading to the so-called mixed-dimensional models which have been the object of many recent works in poromechanics \cite{NEJATI2016123,tchelepi-castelletto-2020,GKT16,GH19,contact-norvegiens,thm-bergen,GDM-poromeca-cont,GDM-poromeca-disc,BDMP:21,BoonNordbotten22}.

Polytopal discretisations are motivated in subsurface applications to cope with the complexity of the geometries representing geological structures including faults/fractures, layering, erosions and heterogeneities.  
Different classes of polytopal methods have been developed in the field of mechanics such as  Discontinuous Galerkin \cite{hansbo-larson}, Hybrid High Order (HHO) \cite{hho-book}, MultiPoint Stress Approximation (MPSA) \cite{keilegavlen-nordbotten}, Hybrid Mimetic Methods \cite{dipietro-lemaire} and Virtual Element Methods (VEM) \cite{beirao-brezzi-marini,da2015virtual}. Some of them have been extended to account for contact-mechanics as in \cite{contact-norvegiens} for the MSPA based on facewise constant approximations of the surface tractions and displacement jump along the fracture network, in \cite{CEP20} for HHO combined with a Nitsche's contact formulation, and in \cite{WRR2016} for VEM based on node to node contact conditions. Among these polytopal methods, VEM, as a natural extension of the Finite Element Method (FEM) to polyhedral meshes, has received a lot of attention in the mechanics community since its introduction in \cite{beirao-brezzi-marini}  and has been applied to various problems including in the context of geomechanics \cite{AHR17}, poromechanics \cite{coulet2020fully,BORIO2021,CRMECA_2023__351_S1_A28_0} and fracture mechanics \cite{Wriggers2024}. 

In \cite{droniou2023bubble}, an  extension of the first order VEM to contact-mechanics is introduced based on a fully discrete framework with vector space of discrete unknowns and reconstruction operators in the spirit of Hybrid High Order discretisations \cite{hho-book}.
Following \cite{Renard03,Hild17,tchelepi-castelletto-2020,BDMP:21} in the FEM case, the contact problem is expressed in mixed form with face-wise constant Lagrange multipliers imposing  the contact conditions in average on each face of the fracture network. 
This approach enables the handling of fracture networks (including corners, tips and intersections), the use of efficient semi-smooth Newton nonlinear solvers, and the preservation at the discrete level of the dissipative properties of the contact terms.
On the other hand, the combination of a first order nodal discretisation of the displacement field with a face-wise constant approximation of the Lagrange multiplier requires a stabilisation to ensure the inf-sup compatibility condition. 
This is achieved in \cite{droniou2023bubble} by extending to the polytopal framework the $\Po^1$-bubble FEM discretisation \cite{Renard03} based on the enrichment of the displacement space by an additional bubble unknown on one side of each fracture face. Numerical evidence is obtained that this enrichment achieves the sought stabilisation, but no proof or analysis thereof is provided in this reference.

The goal of the present work is to precisely perform the numerical analysis of the polytopal discretisation proposed in \cite{droniou2023bubble}. To simplify the presentation, we focus on a static isotropic linear elastic mechanical model with Tresca frictional contact at matrix-fracture interfaces. The key new difficulty is related to the proof of the inf-sup condition between the discrete displacement and Lagrange multiplier spaces, which must account for the polytopal nature of the scheme  and the geometrical complexity of the fracture network including tips, corners and intersections. Previous results \cite{Renard03,Hild17} established for FEM cannot be used since they are adapted neither to the fully discrete framework, nor to fracture networks which necessitates the introduction of a specific $\JSdual$-norm.  
The stability and convergence analysis also requires the proof of a discrete Korn inequality for the discrete displacement field accounting for the bubble stabilisation and the fracture network. The error estimate relies on an abstract estimate taking into account our fully discrete framework and the non-conformity of the displacement discretisation. 

The remaining of this paper is organised as follows. Section \ref{sec:model} introduces the static contact-mechanical model with Tresca friction and its mixed formulation. Section \ref{sec:scheme} recalls the main ingredients of the discretisation from  \cite{droniou2023bubble} with the mesh described in Section \ref{subsec:mesh}, the discrete spaces of displacement and Lagrange multiplier unknowns in Section \ref{subsec:spaces}, the function, jump and gradient reconstruction operators in Section \ref{subsec:operators}, the definition of the interpolation operators in Section \ref{subsec:interpolators},  and the discrete mixed formulation in Section \ref{subsec:mixed}. Section \ref{sec:results} states the main results regarding the well-posedness and convergence analysis of the scheme, while the proofs are reported in Section \ref{sec:proofs}, with additional details in the Appendix. Section \ref{sec:numerics} investigates the numerical behavior of the scheme on analytical solutions in order to assess our theoretical results.

\section{Model}\label{sec:model}
We consider a Discrete Fracture Matrix (DFM) model on the polyhedral domain $\Omega\subset\R^d$ including a fracture network $\Gamma$ with co-dimension 1 defined by 
$$
\overline \Gamma = \bigcup_{i\in I} \overline \Gamma_i. 
$$
We assume that $\Omega\backslash\Gamma$ is connected.
Each fracture $\Gamma_i\subset \Omega$ ($i\in I$) is a polygonal simply connected open subdomain of a plane of $\R^d$. Without restriction of generality, it is assumed that fractures may only intersect at their boundaries. 
The two sides of a given fracture of $\Gamma$ are denoted by $\pm$ in the matrix domain $\Omega \backslash\ov\Gamma$. The two unit normal vectors $\normal^\pm$, oriented outward from the sides $\pm$, satisfy $\normal^+ + \normal^- = \mathbf{0}$. 
Given, for simplicity, homogeneous Dirichlet boundary conditions, the space for the displacement
\begin{equation*}
\U_0 = H^1_0(\O\backslash\overline\Gamma)^d
\end{equation*}
is endowed with the norm $\NORM{\U_0}{\bv} = \NORM{L^2(\Omega\backslash \ov\Gamma)}{\nabla \bv}$ (which is, indeed, a norm since $\Omega\backslash\ov\Gamma$ is connected). 
The oriented jump operator on $\Gamma$ for functions $\bu \in \U_0$ is defined by 
$$
\jump{\bu} = {\gamma^+ \bu - \gamma^- \bu}, 
$$
where $\gamma^\pm$ are the trace operators on both sides of $\Gamma$. 
Its normal and tangential components are denoted respectively by  $\jump{\bu}_{\normal}  = \jump{\bu}\cdot\normal^+$ and $\jump{\bu}_\tang = \jump{\bu} - \jump{\bu}_{\normal} \normal^+$. Note that $\jump{\bu}_\tang$ depends on the orientation, while $\jump{\bu}_{\normal}$ does not.  
The sided normal trace operator on $\Gamma$ oriented  outward to the side $\pm$, applied to $H_\div(\Omega\backslash\ov\Gamma)$ functions, is denoted by $\gamma^{\pm}_{\normal}$.  
The symmetric gradient operator $\bbeps$ is defined on $\U_0$ by $\bbeps(\bv) = {1\over 2} (\nabla \bv + \prescript{t}{}{\nabla \bv})$. 

The model we consider accounts for the mechanical equilibrium equation with a linear elastic constitutive law and a Tresca frictional contact model at matrix--fracture interfaces. In its strong form, it is written
	\begin{equation}
		\label{mode_tresca} 
		\left\{\!\!\!\!
		\begin{array}{lll}
			& -\div \bbsigma(\bu)= \mathbf{f}  & \mbox{ on } \Omega{\backslash}\ov\Gamma,\\[1ex]
			&   \bbsigma(\bu)=2 \mu \bbeps(\bu)+\lambda \div \bu & \mbox{ on }  \Omega{\backslash}\ov\Gamma,\\[1ex]
			& \gamma_\normal^+\bbsigma(\bu)+\gamma_\normal^-\bbsigma(\bu) = \mathbf{0} & \mbox{ on }  \Gamma,\\[1ex]
			& T_{\normal}(\bu) \leqslant 0,~ \jump{\bu }_{\normal} \leqslant 0,~ \jump{\bu }_{\normal}T_{\normal}(\bu) =0 & \mbox{ on }  \Gamma,\\[1ex]
			& \left|\T_{\tang}(\bu)\right| \leqslant \trescacof& \mbox{ on }  \Gamma,\\[1ex]
			& \T_{\tang}(\bu)  \cdot  \jump{\bu }_{\tang} +\trescacof \left| \jump{\bu }_{\tang}\right|=0 & \mbox{ on }  \Gamma, 
		\end{array}
		\right.
	\end{equation}
 with Tresca threshold $\trescacof \geq 0$, Lam\'e coefficients $\mu$ and $\lambda$, and normal and tangential surface tractions $T_{\normal}(\bu) = \gamma_\normal^+\bbsigma(\bu)\cdot \normal^+$ and $\T_{\tang}(\bu) = \gamma_\normal^+\bbsigma(\bu) - T_\normal(\bu) \normal^+$.  It is assumed in the following that the external force term $\mathbf{f}$ belongs to $L^2(\Omega)^d$. 

The weak formulation of the mechanical model  with Tresca frictional-contact \eqref{mode_tresca} is written in mixed form using a vector-valued Lagrange multiplier $\la: \Gamma \to \R^d$ at matrix--fracture interfaces. Define the displacement jump space by
$$
\JS=\left\{\jump{\bv}\,:\,\bv\in\U_0\right\}
$$
and denote by $\JSdual$ its dual space; the duality pairing between these two spaces is written  $\HGbracket{\cdot}{\cdot}$. We will also use this notation for the duality pairing between $\Hhalf$, trace space of $H^1(\Omega\backslash\Gamma)$, and its dual space $\Hhalfdual$. We note that $L^2(\Gamma)^d\subset \JSdual$ and that $\langle \bm{\mu},\bv\rangle_\Gamma=\int_\Gamma \bm{\mu}\cdot\bv$ whenever $\bm{\mu}\in L^2(\Omega)^d$. The dual cone is then defined by
\begin{align*}
	\bm{C}_f = \Big\{\m{} \in \JSdual \,:\,
	\HGbracket{\m}{\bv} \leq \HGbracket{\trescacof}{|\bv_{\tang}|} \mbox{ for all } \bv \in \JS \mbox{ with } \bv\cdot \normal^+ \leq 0\Big\}. 
\end{align*}

The weak mixed-variational formulation of \eqref{mode_tresca} reads: find $\bu \in \U_0$ and $\la \in \bm{C}_f$ such that, for all $\bv \in \U_0$ and $\m \in \bm{C}_f$, 
\begin{subequations}
	\label{Lagrange_meca_contactfriction}
	\begin{align}
		& \dsp \int_\Omega \bbsig( \bu): \bbeps( \bv)  
		+  \HGbracket{\la}{\jump{ \bv}}  \label{Lagrange_meca_contactfriction_1}
		\dsp =  \int_\Omega \mathbf{f}\cdot \bv,\\[.8em]
		& \dsp  \HGbracket{\m - \la}{ \jump{ \bu}}\leq 0.  \label{Lagrange_meca_contactfriction_2} 
	\end{align}
\end{subequations}
Note that, based on the variational formulation, the Lagrange multiplier satisfies
$$
\la = -\gamma_\normal^+\bbsigma(\bu) = \gamma_\normal^-\bbsigma(\bu).
$$

\section{Scheme}\label{sec:scheme}

\subsection{Mesh}\label{subsec:mesh}

We take a polyhedral mesh of the domain $\Omega$ that is conforming with the fracture network $\Gamma$. We also assume that the Tresca threshold $\trescacof\ge 0$ is piecewise constant on the trace of the mesh on $\Gamma$. For each cell $K$ (resp.~face $\sigma$),  we denote by $h_{K}$ (resp.~$h_{\sigma}$) and  $|K|$ (resp.~$|\sigma|$) its diameter and its measure, and we set
$$
h_\D = \max_{K\in \cells} h_K. 
$$
The set of cells $K$, the set of faces $\sigma$,  the set of nodes $s$ and the set of edges $e$ are denoted respectively by  $\cells$, $\faces$, $\nodes$ and $\edges$.  By conformity of the mesh with respect to $\Gamma$, there exists a subset of faces $\faces_{\Gamma} \subset \faces$ such that  
$$  
	\overline{\Gamma} = \bigcup_{\sigma \in \faces_{\Gamma}} \overline{\sigma}. 
$$
We denote by $\cells_{\sigma}$ the set of cells neighboring a face $\sigma \in \faces$; thus, $\cells_{\sigma} = \{ K,L\}$ for interior face $\sigma \in \faces^{\text{int}}$ (in which case we write $\sigma = K|L$) and $\cells_{\sigma} = \{ K\}$ for boundary face $\sigma \in \faces^{\text{ext}}$. Since $\Gamma\subset\Omega$, we have $\faces_{\Gamma} \subset  \faces^{\text{int}}$. For a face $\sigma \in \faces_\Gamma$, $K$ and $L$ in the notation $\sigma=K|L$ are labelled such that $\normal_{\Ksig} = \normal^+$ and $\normal_{\Lsig} = \normal^-$, where $\normal_{\Ksig} $ (resp.~$\normal_{\Lsig}$ ) is the unit  normal vector to $\sigma$ oriented outward of $K$ (resp.~$L$). We denote by $\nodes^{\text{ext}}$ the boundary nodes, and by $\nodes_{\sigma}$ the set of nodes of $\sigma$, $\edges_{\sigma}$ the set of edges of $\sigma$, $\faces_{K}$ the set of faces of $K$, $\nodes_K$ the set of nodes of $K$, and $\cells_s$ the set of cells that contain the vertex $s$. For each $\sigma \in \faces$,  $\normal_\sige$ is the unit normal vector to $e \in \edges_{\sigma}$ in the plane $\sigma$ oriented outward to $\sigma$. For each $K \in \cells$ and $\sigma \in \faces_K$ we denote by $\gamma^{\Ksig}$ the trace operator on $\sigma$ for functions in $H^1(K)$ (or their vector-valued versions). 

Throughout this paper we suppose that mesh regularity assumptions of \cite[Definition 1.9]{hho-book} hold, and we write $a\lsim b$ (resp.~$a\gtrsim b$) as a shorthand for $a\le Cb$ (resp.~$Ca\le b$) with $C>0$ depending only on $\Omega$, $\Gamma$, on the mesh regularity parameter, and possibly on the Lam\'e coefficients and $\textbf{f}$.

If $X\in\cells\cup\faces$ and $\ell\in\N$, we denote by $\Poly{\ell}(X)$ the space of polynomials of degree $\le \ell$ on $X$. For $\mathcal X=\cells$ or $\mathcal X=\faces_\Gamma$, we use the notation $\Poly{\ell}(\mathcal X)$ for the space of piecewise-polynomials of degree $\le \ell$ on $\mathcal X$.

\subsection{Spaces}\label{subsec:spaces}

The degrees of freedom (DOFs) for the displacement are nodal (attached to the vertices of the mesh). To account for the discontinuity of the discrete displacement field at matrix fracture interfaces, these nodal DOFs can be discontinuous across the fracture network -- each vertex on the fracture network (that is not an interior tip) has two or more values attached to it, one for each local connected component of the matrix in a neighborhood of the vertex. To use specific notations, for each $s \in \nodes_\Gamma$ and each $K\in\cells_s$, we denote by $\bv_{\Ks}$ the nodal unknown corresponding to the side of $K$ in the matrix. This unknown is identical for all cell lying on the same side as $K$: if $L\in\cells_s$ is on the same side as $K$, then $\bv_\Ks=\bv_\Ls$. The notation $\Ks$ is used for the set of all cells containing $s$ and lying on the same side of $\Gamma$ as $K$. If $s\not\in\nodes_\Gamma$, there is a unique displacement unknown at $s$, which is denoted by $\bv_\Ks$; in that case, $\Ks=\cells_s$. 
Figure \ref{dof_3} visually illustrates this idea. 

Additionally, on the positive side of the fracture, we attach to each face a vector DOF corresponding to a ``correction'' of the face values generated by the nodal DOFs; this correction plays the role, at the discrete level represented by the space of DOFs, of a bubble function and is introduced to ensure a suitable inf--sup condition between the space of Lagrange multipliers and the jump of the displacements.

\begin{figure}[H]
	\begin{center}
		\includegraphics[width=9.4cm]{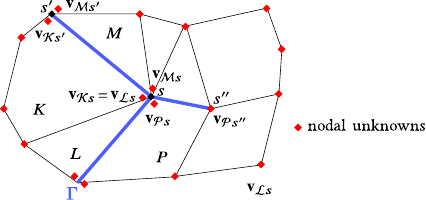}
		\caption{Nodal degrees of freedom.}\label{dof_3}
	\end{center}
\end{figure}

For each $K\in\cells$, we denote by
\begin{equation*}
	\faces_{\Gamma,K}^+ = 
	\Big\{ \sigma \in \faces_{\Gamma} \cap \faces_{K}\,:\, \normal_{\Ksig} \cdot \normal^+  > 0  \Big\}
\end{equation*}
the set of faces of $K$ that lie on the positive side of $\Gamma$ (that set is empty if $K$ does not touch $\Gamma$ or only touches it from its negative side). According to the discussion above, the discrete space of displacement is
\begin{equation}\label{UD_dof}
  \begin{aligned}
  \UDz = \Big\{\bv_\D={}&((\bv_{\Ks})_{K\in\cells,\,s\in\nodes_K},(\bv_{\Ksig})_{K\in\cells,\,\sigma\in\mathcal F_{\Gamma,K}^+})\,:\\
  &\bv_\Ks\in\R^d\,,\;\bv_\Ksig\in\R^d\,,\;\bv_\Ks=0\mbox{ if $s\in\nodes^{\text{ext}}$}\\
  &\bv_\Ks=\bv_\Ls\mbox{ if $K,L \in\cells_s$ are on the same side of $\Gamma$}\Big\}.
  \end{aligned}
\end{equation}
The Lagrange multiplier plays the role of approximations of $-\gamma_\normal^+\bbsigma(\bu)$ on $\Gamma$. Its space is made of piecewise constant vectors:
$$ 
  \MD = \big\{  \la_{\D} \in L^2\left(  \Gamma \right)^d\,:\, \la_\sigma:=(\la_{\D})_{|\sigma}\mbox{ is constant for all $\sigma\in\faces_\Gamma$} \big\}.
$$
  For $\la_{\D} \in \MD$, we define its normal and tangential components by
  $$
  \lambda_{\D,\normal} = \la_{\D} \cdot \normal^+, \quad \la_{\D,\tang} = \la_\D - \lambda_{\D,\normal}  \normal^+, 
  $$
   and the discrete dual cone by
$$
\CD = \big\{ \la_{\D} \in \MD \,:\, \lambda_{\D,\normal} \ge 0,\, |\la_{\D,\tang}| \leq \trescacof \big\} \subset \bm{C}_f.
$$
(the inclusion being easy to check).

\subsection{Reconstruction operators in $\UDz$}\label{subsec:operators}

We first define, for each $K\in\cells$ and $\sigma\in\faces_K$, a tangential face gradient $\nabla^\Ksig:\UDz\to \Poly{0}(\sigma)^{d\times d}$ and tangential displacement reconstruction $\Pi^{\Ksig}:\UDz\to\Poly{1}(\sigma)^d$. First, we choose nonnegative weights $(\omega_s^\sigma)_{s\in\nodes_\sigma}$ to express the center of mass $\overline{\mathbf{x}}_\sigma$ of $\sigma$ in terms of its vertices:
\begin{equation}\label{eq:choice.weights.sigma}
\overline{\mathbf{x}}_{\sigma} = \sum_{s \in \nodes_{\sigma} } \omega_{s}^{\sigma} \x_s\,,\quad
\sum_{s \in \nodes_{\sigma} } \omega_{s}^{\sigma}=1.
\end{equation}
Then, for $\bv_\D\in\UDz$, we set
\begin{equation}\label{def:def.operators.sigma}
\begin{aligned}
  \nabla^{\Ksig} \bv_{\D} ={}& \frac{1}{|\sigma|} \sum_{e= s_1 s_2 \in \edges_{\sigma}} |e| {\mathbf{v}_{\Ks_1} + \mathbf{v}_{\Ks_2} \over 2} \otimes \normal_\sige,\\
  \Pi^{\Ksig} \bv_{\D}(\mathbf{x})  ={}& \nabla^{\Ksig} \bv_{\D} (\mathbf{x} - \overline{\mathbf{x}}_{\sigma}) + \overline{\bv}_{\Ksig}\quad\forall\mathbf{x}\in\sigma,\quad\mbox{ where } \overline{\bv}_{\Ksig}=\sum_{s \in \nodes_{\sigma} } \omega_{s}^{\sigma}\bv_{\Ks}.
\end{aligned}
\end{equation}
Above, we have noted $e=s_1s_2$ to indicate that the edge $e$ has vertices $s_1$, $s_2$.

If $\sigma \in \faces_{\Gamma}$ is a fracture face, and $K$ (resp.~$L$) is the cell on the positive (resp.~negative) side of $\sigma$, we define the displacement jump operator on $\sigma$ as $\jump{{\cdot}}_\sigma:\UDz\to\Poly{0}(\sigma)^d$ such that, for all $\bv_\D\in\UDz$,
\begin{equation}\label{eq:def.jump.sigma}
  \jump{\bv_{\D}}_{\sigma}  =  \frac{1}{|\sigma|} \int_{\sigma}(\Pi^{\Ksig}  \bv_{\D} - \Pi^{\Lsig} \bv_{\D}) +  \bv_{\Ksig}
  = \overline{\bv}_\Ksig-\overline{\bv}_\Lsig +  \bv_{\Ksig}.
\end{equation}
The normal component of that jump is denoted by $\jump{{\cdot}}_{\sigma,\normal} =  \jump{{\cdot}}_{\sigma} \cdot \mathbf{n}_{\Ksig}$.

For each cell $K\in\cells$, we select nonnegative weights $(\omega_s^K)_{s\in\nodes_K}$ of a linear decomposition of the center of mass $\overline{\mathbf{x}}_K$ of $K$ in terms of its vertices
$$
  \overline{\mathbf{x}}_K=\sum_{s\in\nodes_K}\omega_s^K\mathbf{x}_s\,,\quad 
  \sum_{s \in \nodes_K } \omega_{s}^K=1,
$$
and we design a gradient reconstruction $\nabla^K:\UDz\to\Poly{0}(K)^{d\times d}$ and a displacement reconstruction $\Pi^K:\UDz\to\Poly{1}(K)^d$ by setting, for $\bv_\D\in\UDz$,
\begin{align}
  \label{eq:def.nablaK}
  \nabla^{K} \bv_{\D}  ={}& \frac{1}{|K|} \sum_{\sigma \in \faces_{K}}  |\sigma| \overline{\bv}_{\Ksig} \otimes \normal_\Ksig + \sum_{\sigma \in \faces_{\Gamma,K}^+}  \frac{|\sigma|}{|K|} \bv_{\Ksig} \otimes \normal_\Ksig,\\
  \label{eq:def.PiK}
  \Pi^{K} \bv_{\D}(\mathbf{x})  ={}& \nabla^{K} \bv_\D (\mathbf{x} - \overline{\mathbf{x}}_{K}) + \overline{\bv}_{K}\quad\forall\mathbf{x}\in K\,,\quad
  \mbox{ where }\overline{\bv}_{K} = \sum_{s \in \nodes_K } \omega_{s}^K \bv_{\Ks}.
\end{align}

These local jump, gradient and displacement reconstructions are patched together to create their global piecewise polynomial counterparts $\jump{{\cdot}}_\D:\UDz\to \Poly{0}(\faces_\Gamma)^d$, $\nabla^\D:\UDz\to\Poly{0}(\cells)^{d\times d}$ and $\Pi^\D:\UDz\to\Poly{1}(\cells)^d$: for all $\bv_\D\in\UDz$,
\begin{align*}
(\jump{\bv_\D}_{\D})_{|\sigma}={}&\jump{\bv_\D}_\sigma\qquad\forall\sigma\in\faces_\Gamma,\\
(\nabla^\D\bv_\D)_{|K}={}&\nabla^K\bv_\D\qquad\forall K\in\cells,\\
(\Pi^\D\bv_\D)_{|K}={}&\Pi^K\bv_\D\qquad\forall K\in\cells.
\end{align*}
We also define the cellwise constant reconstruction operator 
$\wt{\Pi}^{\D}\bv_{\D}:\UDz\to\Poly{0}(\cells)^d$ such that  $({\wt{\Pi}^{\D}\bv_{\D}})_{|K} = \overline{\bv}_K.$
Finally, the discrete symmetric gradient $\bbeps_\D$, divergence $\div_\D$ and stress tensor $\bbsigma_\D$ are deduced from the previous operators:
$$
\bbeps_{\D} =\frac{1}{2}(\nabla^{\D}+ \prescript{T}{}{\nabla^{\D}}),\quad \div_{\D} = \text{Tr}\left( \bbeps_{\D}\right) \quad \text{and}\quad \bbsigma_{\D}(\cdot) = 2\mu\bbeps_{\D}(\cdot) + \lambda \div_{\D}(\cdot) \mathbb{I}.
$$

\begin{remark}[Non planar faces]\label{non-planar}
	At this point, we have only considered meshes with planar faces. In numerical simulations in geosciences, meshes with non-planar faces are however naturally encountered, for example in Corner Pointed Geometries (CPG) situations. One approach to handle such meshes is to cut the non-planar faces into two or more planar subfaces; this leads to polytopal meshes that can be handled by our method. Another approach was introduced in \cite{CHI2017148}. It consists in adding a barycentric center $\mathbf{x}_{\sigma} =  \frac{1}{\sharp \nodes_{\sigma}} \sum_{s \in \nodes_{\sigma}} \mathbf{x}_{s} $ to the non-planar face $\sigma$, which is artificial in the sense that it is not counted as a geometric node of the mesh. The discrete displacement $\bu_\sigma$ at $\mathbf{x}_{\sigma}$ is then defined by barycentric linear combination of the nodal displacements at the vertices of the face. Additionally, a triangulation of the face $\sigma$, centered at $\mathbf{x}_{\sigma}$, is defined. In this context, the gradient reconstruction operator \eqref{eq:def.nablaK} in each cell must be adapted as follows 	
	\begin{equation*}
		\nabla^{K} \bv_{\D} = \frac{1}{|K|} \sum_{\sigma \in \faces_{K}} \sum_{e=s_1 s_2 \in \edges_{\sigma}} \frac{|T_e|}{3} (\bv_{\Ks_1} + \bv_{\Ks_2} + \bv_{\sigma}) \otimes \normal_{KT_e} + \sum_{\sigma \in \faces_{\Gamma,K}^+} \frac{|\sigma|}{|K|} \bv_{\Ksig} \otimes \normal_\Ksig,
	\end{equation*}
	where $T_e$ is the triangle defined by $\mathbf{x}_{\sigma}$ and the edge $e$, $\normal_{KT_e}$ is the normal vector to triangle $T_e$ pointing out of $K$, and $\bv_{\sigma}= \frac{1}{\sharp \nodes_{\sigma}} \sum_{s \in \nodes_{\sigma}} \bv_{\Ks}$.
Note that the fracture faces are still assumed to be planar, hence the bubble terms are unchanged. Moreover, for cells with planar faces, this new gradient is identical to \eqref{eq:def.nablaK}. 
        With this approach, the same discrete space of displacement $\UDz$ is retained, as the unknowns remain unchanged. 
\end{remark}

\subsection{Interpolators}\label{subsec:interpolators}

The space $\mathcal{C}^0_0(\overline{\Omega}\backslash \Gamma)$ is spanned by functions that are continuous on $\overline{\Omega}\backslash\Gamma$, have limits on each side of $\Gamma$, and vanish on $\partial\Omega$. The interpolator $\IUD:\mathcal C^0_0(\overline{\Omega}\backslash\Gamma)^d\to\UDz$ is defined through its components by setting, for $\bv\in \mathcal C^0_0(\overline{\Omega}\backslash\Gamma)^d$,
\begin{equation}\label{eq:def.ID}
  \begin{aligned}
  (\IUD \bv )_{\Ks}  ={}& \bv_{|K}(\mathbf{x}_s)&&\quad\forall K\in\cells\,,\;\forall s \in \nodes_K,\\
  (\IUD \bv)_{\Ksig} ={}& \frac{1}{|\sigma|} \int_{\sigma} ( \gamma^{\Ksig} \bv - \Pi^{\Ksig}(\IUD\bv) ) &&\quad\forall K\in\cells\,,\;\forall\sigma \in \faces_{\Gamma,K}^+.
  \end{aligned}
\end{equation}
This definition is seemingly recursive, since the DOF corresponding to $\Ksig$ are built using the interpolator $\IUD$ itself. However, the definition \eqref{def:def.operators.sigma} of $\Pi^\Ksig$ shows that $\Pi^{\Ksig}(\IUD\bv)$ only depends on the nodes unknowns $(\IUD\bv)_{\Ks}$, which are properly defined without any self-reference to $\IUD$. We also note that, since $\bv=0$ on $\partial\Omega$, this operator indeed defines an element in $\UDz$.

The interpolator $\IMD:L^2(\Gamma)\to \MD$ on the Lagrange space simply corresponds to averaging on each face: for $\la\in L^2(\Gamma)^d$,
$$
  (\IMD\la)_{\sigma}=\frac{1}{|\sigma|}\int_\sigma\la\qquad\forall\sigma\in\faces_\Gamma.
$$
\subsection{Mixed-variational formulation}\label{subsec:mixed}

We now introduce the numerical scheme for the mixed-variational formulation of the mechanics contact problem  \eqref{Lagrange_meca_contactfriction}: Find $( \bu_{\D}, \la_{\D} ) \in \UDz \times \CD$
such that, for all $( \bv_{\D}, \bs{\mu}_{\D} ) \in \UDz \times \CD$,
  \begin{subequations}
  	\label{mixed_discrete}
  	\begin{align}
  		& \int_{\Omega} \bbsigma_{\D} (\bu_{\D}): \bbeps_{\D}\left(\bv_{\D} \right)  + S_{\mu,\lambda,\D}\left( \bu_{\D}, \bv_{\D} \right) + \int_{\Gamma} \la_{\D} \cdot \jump{\bv_{\D}}_{\D} = \int_\Omega  \mathbf{f} \cdot \wt \Pi^{\D} \bv_{\D}, \label{eq:meca.var.tresca}  \\[2.5ex]
  	& \int_{\Gamma} \left( \mu_{\D} - \la_{\D} \right) \cdot \jump{\bu_{\D}}_{\D}  \le  0. \label{eq:meca.ineq.contact.tresca}
  	\end{align}
  \end{subequations}
Here, $S_{\mu,\lambda,\D}$ is the scaled stabilisation bilinear form defined by
\begin{equation*}
		 S_{\mu,\lambda,\D}( \bu_{\D}, \bv_{\D}) = \sum_{K\in\cells}(2 \mu_K+\lambda_K)S_K(\bu_\D,\bv_D)
\end{equation*}
where, for each $K\in\cells$, $\mu_K$ and $\lambda_K$ are respectively the average of $\lambda$ and $\mu$ on $K$ and the local stabilisation bilinear form $S_K:\UDz\times\UDz \to\R$ is given by
\begin{equation}\label{eq:def.SK}
\begin{aligned}
	S_K(\bu_\D,\bv_\D)={}& h_{K}^{d-2}  \sum_{s \in \nodes_K} \left(\bu_{\Ks} -\Pi^{K}\bu_{\D} (\mathbf{x}_s) \right) \cdot \left( \bv_{\Ks} -\Pi^{K}\bv_{\D}  (\mathbf{x}_s) \right) \\
		 &+ h_{K}^{d-2}\sum_{\sigma \in \faces_{\Gamma,K}^+} \bu_{\Ksig}  \cdot  \bv_{\Ksig}.
\end{aligned}
\end{equation}

Let us also introduce the unscaled stabilisation bilinear form  
\begin{equation*}
		 S_{\D}( \bu_{\D}, \bv_{\D}) = \sum_{K\in\cells} S_K(\bu_\D,\bv_D). 
\end{equation*}

\begin{remark}
Thanks to the fracture face-wise Lagrange multiplier, the variational inequality  \eqref{eq:meca.ineq.contact.tresca} together with $\la_{\D} \in \CD$ can be equivalently replaced by the following non linear equations (see \cite{beaude2023mixed}), in which $[s]_{\R^+}=\max(0,s)$, and $[\boldsymbol{\xi}]_\trescacof$ is the projection of $\boldsymbol{\xi}\in\R^d$ on the ball of center $0$ and radius $\trescacof$:
\begin{equation}\label{coupling_law}
	\left\{\!\!
	\begin{array}{ll}
		& \lambda_{\D,\normal} = \[ \lambda_{\D,\normal} + \beta_{\D,\normal} \jump{\bu_{\D}}_{\D,\normal} \]_{\mathbb{R}^+} \\[3ex]
		& \la_{\D,\tang} = \[ \la_{\D,\tang} + \beta_{\D,\tang} \jump{\bu_{\D}}_{\D,\tang} \]_{\trescacof}
	\end{array}
	\right.
\end{equation}
(these equations can easily be expressed locally to each fracture face). Here, $\beta_{\D,\normal} > 0$, $\beta_{\D,\tang} > 0$ are given face-wise constant functions along $\Gamma$. 
\end{remark}
  
\begin{remark}[Virtual element interpretation]
It is shown in \cite{droniou2023bubble} that the scheme \eqref{mixed_discrete} can be re-interpreted in a Virtual Element presentation.
\end{remark}

\section{Main results}\label{sec:results}

To carry out the convergence analysis of the scheme, we need to define norms on the spaces of unknowns: an $H^1$-like norm on the space of displacement and an $\JSdual$-like norm on the space of Lagrange multipliers.

\begin{definition}[Discrete $H^1$-like semi-norm on $\UD$]
The semi-norm $\NORM{1,\D}{{\cdot}}$ on $\UD$ is defined by: for all $\bu_\D\in \UD$,
\begin{equation}\label{eq:def.normD}
\NORM{1,\D}{\bu_\D}:= \left(\sum_{K\in\cells}\NORM{1,K}{\bu_\D}^2\right)^{1/2}
\mbox{ with } \NORM{1,K}{\bu_\D}=\left(\NORM{L^2(K)}{\nabla^K\bu_\D}^2+S_K(\bu_\D,\bu_\D)\right)^{1/2},
\end{equation}
where $S_K$ is given by \eqref{eq:def.SK}. We note that $\NORM{1,\D}{{\cdot}}$ is genuinely a norm on $\UDz$ since $\Omega\backslash\ov\Gamma$ is connected.
\end{definition}

To define the norm on $\MD$, we recall that $\Gamma=\cup_{i\in I}\Gamma_i$ with each $\Gamma_i$ open connected subset of a hyperplane, and $\Gamma_i\cap\Gamma_j=\emptyset$ if $i\not=j$. We define $\Omega_i^+$ as the intersection of $\Omega$ with the half-plane defined by $\Gamma_i$ and $(\normal^+)_{|\Gamma_i}$; see Figure \ref{fig:split-Gamma}. The space $H^1(\Omega_i^+;\Gamma_i)$ is spanned by functions in $H^1(\Omega_i^+)$ that vanish on $\partial\Omega_i^+\backslash \Gamma_i$.

\begin{figure}[h!]
\centering
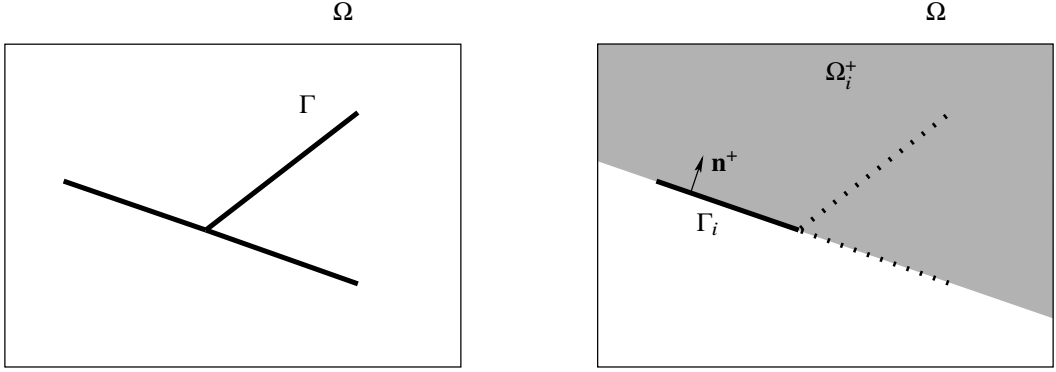
\caption{Splitting of the fracture network and construction of sub-domains used to define
the norm on $\MD$.}
\label{fig:split-Gamma}
\end{figure}

\begin{definition}[$\JSdual$-like norm on $\MD$]
The $\JSdual$-like norm on $\MD$ is defined by: for all $\la_\D\in \MD$,
\begin{equation}
  \NORM{-\nicefrac12,\Gamma}{\la_\D}= \sum_{i\in I}\NORM{-\nicefrac12,\Gamma_i}{\la_\D}
  \ \mbox{ with }\ 
  \NORM{-\nicefrac12,\Gamma_i}{\la_\D}=\sup_{\bv_i\in H^1(\Omega_i^+;\Gamma_i)^d\backslash\{0\}}
  \frac{\int_{\Gamma_i}\la_\D\cdot\bv_i}{\NORM{H^1(\Omega_i^+)}{\bv_i}}.
\label{eq:def.H12.norm}
\end{equation}
\end{definition}

\begin{remark}[Norm on the fracture network] 
The definition of this $\JSdual$-like norm is non-standard due to localisation to each planar component of $\Gamma$. This is however required to properly take into account the possible complex topology of the fracture network (triple -- or more -- intersections of fractures, etc.) and the fact that the inf-sup condition below is based on the jump of the functions, not their traces. See Remark \ref{rem:jump.norm} for more insight on this.

In case of a simple network in which no more than two planar fractures intersect at a given location, since these intersections (and the fracture tips) have a zero $2$-capacity it can be checked that the norm \eqref{eq:def.H12.norm} is equivalent to a more standard $\Hhalfdual$-norm on the network. 
\end{remark}

\begin{proposition}[Existence and uniqueness result]\label{wellposedness}
There exists a unique solution $(\bu_{\D}, \la_{\D}) \in \UDz \times \CD$ to \eqref{mixed_discrete}. 
\end{proposition}

\begin{proof}
\cite[Theorem 3.9]{HASLINGER1996313} gives the existence of a solution. The uniqueness of $\bu_\D$ derives from the discrete Korn inequality \eqref{eq:Korn} below and the uniqueness of $\la_\D$ from the discrete inf-sup property \eqref{eq:inf-sup} below. 
\end{proof}

To state the error estimates, we introduce the following notations:
\begin{itemize}
\item $\ProjFG$ is the orthogonal projection on $\mathbb{P}^0(\faces_\Gamma)$.
\item The (primal) consistency error is: for $\bv_\D\in\UDz$,
\begin{equation}\label{def:CD}
	C_{\D}(\bu,\bv_{\D}) = \(\NORM{L^2(\Omega\backslash\ov\Gamma)}{\grad \bu - \grad^{\D} \bv_{\D}}^2 +  S_{\D}( \bv_{\D},  \bv_{\D})) \)^{1/2}.
\end{equation}
\item Letting
\begin{equation}\label{def:esp.strain}
\mathbf{W} = \big\{  \bbsigma \in H_{\div}(\Omega\backslash\ov\Gamma; \mathcal{S}^d(\R)),\,:\, \gamma_\normal^+ \bbsigma +  \gamma_\normal^-\bbsigma = \mathbf{0},\, \gamma_\normal^+ \bbsigma \in L^2(\Gamma)^d \big\}, 
\end{equation}
(where $\mathcal{S}^d(\R)$ is the space of symmetric $d\times d$ matrices with real coefficients), 
the adjoint consistency error (or limit-conformity measure) is defined, for $\bbsigma\in\mathbf{W}$, by
\begin{equation}\label{def:WD}
\begin{aligned}
\mathcal{W}_{\D}(\bbsigma) ={}& \sup_{\bv_\D\in\UDz}\frac{w_{\D}(\bbsigma,\bv_{\D})}{\NORM{1,\D}{\bv_\D}},\\
\mbox{where } w_{\D}(\bbsigma,\bv_{\D}) ={}& 	-\int_{\Omega} \bbsigma: \bbeps_{\D}(\bv_{\D} )  + \int_{\Gamma} \gamma_\normal^+ \bbsigma \cdot \jump{\bv_{\D}}_{\D} - {}  \int_{\Omega}   \wt{\Pi}^{\D}\bv_{\D}\cdot \div \bbsigma.
\end{aligned}
\end{equation}
\item The discrete $\Hhalf$-norm is defined on $L^2(\Gamma)^d$ by: for $\m \in L^2(\Gamma)^d$,
\begin{equation}\label{def:norm.1/2D}
	 \NORM{\nicefrac 12,\D}{\m} = \(\sum_{\sigma \in \faces_{\Gamma}} h_{\sigma}^{-1} \NORM{L^2(\sigma)}{\m}^2\)^{1/2}. 
\end{equation}
\end{itemize}

\begin{theorem}[Abstract error estimate]\label{abstr.err.est}
For $(\bu_{\D}, \la_{\D})$ solution of \eqref{mixed_discrete} and $(\bu, \la)$ solution of \eqref{Lagrange_meca_contactfriction} with $\la \in L^2(\Gamma)^d$, we have the following abstract error estimate	
\begin{equation}\label{estimat_error}
\begin{aligned}
	\NORM{L^2(\Omega\backslash\ov\Gamma)}{\grad^{\D}\bu_\D - \grad \bu} +	{}&\NORM{-\nicefrac 12,\Gamma}{\la_\D - \la} \lsim  \NORM{-\nicefrac 12,\Gamma}{\la - \ProjFG \la}  + \mathcal{W}_{\D}(\bbsigma(\bu)) \\
	 + \inf_{\bv_{\D} \in \UDz } \bigg\{ \Big( &\NORM{L^2(\Gamma)}{\la - \ProjFG \la}\NORM{L^2(\Gamma)}{\jump{\bv_{\D}}_{\D} - \jump{\bu}}\Big)^{1/2} \\
	& + \NORM{\nicefrac 12,\D}{\jump{\bv_{\D}}_{\D} - \ProjFG\jump{\bu}}   + C_{\D}(\bu,\bv_{\D})\bigg\}. 
\end{aligned}
\end{equation}
\end{theorem}

\begin{proof}
See Section \ref{sec:proof.abstract.error}.
\end{proof}

In the following theorem, we denote by $H^2(\cells)$ (resp.~$H^1(\faces_\Gamma)$) the space of functions defined on $\Omega$ that are $H^2$ on each $K\in\cells$ (resp.~defined on $\Gamma$ and $H^1$ on each $\sigma\in\faces_\Gamma$). These spaces are endowed with their usual broken semi-norms. 

\begin{theorem}[Error estimate]\label{error_estimate}
 Let $(\bu,\la)$ be the solution to \eqref{Lagrange_meca_contactfriction} and assume that $\bu\in H^2(\cells)$ and $\la\in H^1(\faces_\Gamma)$. Then the solution $(\bu_{\D}, \la_{\D})$ of \eqref{mixed_discrete} satisfies the following error estimate:
$$
	\NORM{L^2(\Omega\backslash\ov\Gamma)}{\grad^{\D}\bu_\D - \grad \bu} +	\NORM{-\nicefrac 12,\Gamma}{\la_\D - \la} \lsim h_\D 
	\(\SEMINORM{H^1(\faces_\Gamma)}{\la} + \SEMINORM{H^2(\cells)}{\bu} + \SEMINORM{H^1(\faces_\Gamma)}{\jump{\bu}}\). 
$$
\end{theorem}

\begin{proof}
See Section \ref{sec:proof.error}.
\end{proof}

\section{Proof of the error estimate}\label{sec:proofs}

\subsection{Inf-sup condition}

We prove in this section the following result, which establishes that the bubble degree of freedom in \eqref{UD_dof} plays its role controlling the Lagrange multiplier through the jump of the displacements.

\begin{theorem}[Discrete inf-sup condition]\label{th:infsup}
It holds
\begin{equation}\label{eq:inf-sup}
\sup_{\bv_\D\in\UDz\backslash\{0\}}\frac{\int_\Gamma \la_\D\cdot\jump{\bv_\D}_\D}{\NORM{1,\D}{\bv_\D}}\gtrsim \NORM{-\nicefrac12,\Gamma}{\la_\D}\qquad\forall \la_\D\in \MD. 
\end{equation}
\end{theorem}

\subsubsection{Fracture-compatible averaged interpolator}

To prove the inf-sup condition, we need a Cl\'ement-like interpolator that takes into account the fracture network. Due to the design of the $\JSdual$-like norm, we will use this interpolator considering only one planar component at a time. In the following, we therefore fix $i\in I$ and, in the constructions below, the sets $\Ksi$ are considered only in respect to $\Gamma_i$. 
Thus, for all $K\in\cells$ and $s\in\nodes_K$, $\Ksi$ is the set of all cells $L\in\cells_s$ that lie on the same side of $\Gamma_i$ as $K$. In particular, $\Ks\subset\Ksi$ with equality whenever $s$ is internal to $\Gamma_i$.

For each $K\in\cells$ and $s\in\nodes_K$, we take an open set $U_{\Ksi}\subset \bigcup_{K\in \Ksi}K$ and a function $\varpi_{\Ksi}\in L^\infty(U_{\Ksi})$ such that
\begin{subequations}\label{eq:def.varpi}
\begin{align}
\label{eq:def.varpi.bounds}
&|U_{\Ksi}|\gtrsim \max_{K\in\Ksi}|K|\,,\quad |\varpi_{\Ksi}|\lsim 1\,,\\
\label{eq:def.varpi.integral}
&\frac{1}{|U_{\Ksi}|}\int_{U_{\Ksi}}\varpi_{\Ksi}=1\,,\quad \frac{1}{|U_{\Ksi}|}\int_{U_{\Ksi}}\mathbf{x}\varpi_{\Ksi}=\mathbf{x}_s.
\end{align}
\end{subequations}

 If $s\not\in\Gamma$, $U_{\Ksi}$ can be a taken as a ball centered at $s$, and $\varpi_{\Ksi}=1$. Figure \ref{fig:UKsi} illustrates possible choices for $U_{\Ksi}$ depending on the nature of $s$, and Appendix \ref{sec:existence.varpi} presents an explicit way to construct $(U_{\Ksi},\varpi_{\Ksi})$ in the generic case. 

\begin{figure}[h!]
\centering
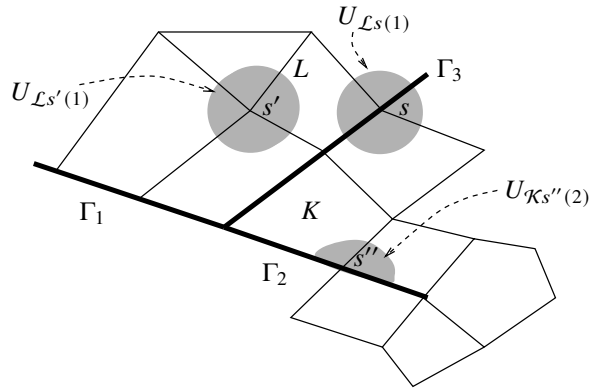
\caption{Domains for local fracture-compatible averages.}
\label{fig:UKsi}
\end{figure}

The space $H^1_0(\Omega\backslash\Gamma_i)$ is the subspace of $H^1(\Omega\backslash\Gamma_i)$ spanned by functions that vanish on $\partial\Omega$, but not necessarily on $\Gamma_i$. For $\bv\in H^1_0(\Omega\backslash\Gamma_i)^d$, the averaged interpolate $\IUD^{i,\ave}\bv\in\UDz$ of $\bv$ is defined by
\begin{subequations}\label{eq:def.IDave}
\begin{alignat}{4}
  (\IUD^{i,\ave} \bv )_{\Ks}  ={}& 0&&\quad \mbox{ for all }s \in \nodes^{\text{ext}},\label{eq:IUDave.bdry}\\
	(\IUD^{i,\ave} \bv )_{\Ks}  ={}& \frac{1}{|U_{\Ksi}|}\int_{U_{\Ksi}} \varpi_{\Ksi}\bv&&\quad \mbox{ for all }K\in\cells\,,\;s \in \nodes_K\setminus\nodes^{\text{ext}},\label{eq:IUDave.node}\\
	\label{eq:IDiave.sigma.Gammai}
	(\IUD^{i,\ave} \bv )_{\Ksig} ={}& \frac{1}{|\sigma|} \int_{\sigma} ( \gamma^{\Ksig}{\bv} - \Pi^{\Ksig}(\IUD^{i,\ave}\bv) ) &&\quad \mbox{ for all }K \in \cells\,,\; \sigma \in \faces_{\Gamma_i,K}^{+},\\
	\label{eq:IDiave.sigma.notGammai}
	(\IUD^{i,\ave} \bv )_{\Ksig} ={}& 0 &&\quad \mbox{ for all }K \in \cells\,,\; \sigma \in \faces_{\Gamma,K}^+\setminus\faces_{\Gamma_i,K}^{+}.
\end{alignat}
\end{subequations}

\begin{remark}[Averaged interpolator]\label{rem:IDave}
The usage of averages in the definition of the nodal values is mandatory since we need well-defined and stable interpolations of functions with only $H^1$-regularity -- see Proposition \ref{prop:stability.IDave}. However, the averages must not be done across $\Gamma_i$ (hence the condition $U_{\Ksi}\subset \bigcup_{K\in \Ksi}K$), and must have $\mathbf{x}_s$ as center of mass (hence the integral conditions on $\varpi_{\Ksi}$), to ensure that the interpolator is exact on linear functions -- see the proof of Proposition \ref{prop:stability.IDave}.

If $s\in\Gamma$ is not in the relative interior of $\Gamma_i$, then for all $K,L\in\cells_s$ we have $\Ksi=\Lsi$ and thus $(\IUD^{i,\ave}\bv)_{\Ks}=(\IUD^{i,\ave}\bv)_{\Ls}$. In other words, this $\Gamma_i$-adapted interpolator $\IUD^{i,\ave}$ produces single nodal values on $\Gamma_j$ for $j\not=i$, and possibly multiple nodal values only on vertices in the interior of $\Gamma_i$. This is coherent with the continuity properties of functions in $H^1_0(\Omega\backslash\Gamma_i)$ that it interpolates.

Finally, we note that a zero value is imposed for boundary nodes, so that $\IUD^{i,\ave}\bv\in \UDz$ (computing boundary nodal values by averaging would not ensure that they vanish), and that the ``bubble'' value is set at zero for faces not on $\Gamma_i$. 
\end{remark}

The main properties of this interpolator are its behaviour with respect to the jump, and its $H^1$-stability.

\begin{proposition}[Jump of the averaged interpolate]
For all $\bv\in H^1_0(\Omega\backslash \Gamma_i)^d$, it holds
\begin{equation}\label{eq:trace.interpolate}
\forall K\in\cells\,,\;\forall\sigma\in\faces_K\,,\quad
\int_{\sigma} \jump{\IUD^{i,\ave}\bv}_\sigma=\left\{
\begin{array}{ll}
\int_\sigma (\gamma^{\Ksig}{\bv}-\Pi^{\Lsig}\IUD^{i,\ave}\bv)&\mbox{ if }\sigma\in \faces_{\Gamma_i,K}^{+},\\
0&\mbox{ if }\sigma\in \faces_{\Gamma,K}^+\setminus \faces_{\Gamma_i,K}^{+}.
\end{array}\right.
\end{equation}
\end{proposition}

\begin{proof}
If $\sigma\in\faces_{\Gamma_i,K}^{+}$, then the equality directly follows from the definition \eqref{eq:def.jump.sigma} of the jump on $\sigma$ and from \eqref{eq:IDiave.sigma.Gammai}.
If $\sigma\in \faces_{\Gamma,K}^+\setminus \faces_{\Gamma_i,K}^{+}$ then any $s\in\nodes_\sigma$ is not internal to $\Gamma_i$ and thus, as noted in Remark \ref{rem:IDave}, for all $K,L\in\cells_s$ we have $(\IUD^{i,\ave}\bv)_{\Ks}=(\IUD^{i,\ave}\bv)_{\Ls}$. Hence, $\Pi^{\Ksig}\IUD^{i,\ave}\bv=\Pi^{L_\sigma}\IUD^{i,\ave}\bv$. The definition \eqref{eq:def.jump.sigma} of the jump on $\sigma$, together with the fact that $(\IUD^{i,\ave}\bv)_{\Ksig}=0$ by \eqref{eq:IDiave.sigma.notGammai} yields the relation in \eqref{eq:trace.interpolate}.
\end{proof}

Establishing the $H^1$-stability of $\IUD^{i,\ave}$ requires the following lemma.

\begin{lemma}[DOF-based bound on the discrete norm]\label{lem:bound.DOF}
Let $K\in\cells$. Recalling the definition \eqref{eq:def.normD} of $\NORM{1,K}{{\cdot}}$, we have, for all $\bw_K=((\bw_\Ks)_{s\in\nodes_K},(\bw_\Ksig)_{\sigma\in\faces_{\Gamma,K}^+})$,
\begin{equation}\label{eq:bound.normK}
\NORM{1,K}{\bw_K}\lsim h_K^{-1}|K|^{1/2}\Big(\max_{s\in\nodes_K}|\bw_{\Ks}| + \max_{\sigma\in\faces_{\Gamma,K}^+}|\bw_{\Ksig}|\Big).
\end{equation}
\end{lemma}

\begin{proof}
Set 
$$
|\bw_K|_{\infty,K}=\max_{s\in\nodes_K}|\bw_{\Ks}| + \max_{\sigma\in\faces_{\Gamma,K}^+}|\bw_{\Ksig}|.
$$
The definition \eqref{eq:def.nablaK} of $\nabla^K$ and the mesh regularity assumption (which ensures that $|\sigma|/|K|\lsim h_K^{-1}$) shows that
\begin{equation}\label{eq:bound.nablaK}
\NORM{L^\infty(K)}{\nabla^K\bw_K}\lsim h_K^{-1}|\bw_K|_{\infty,K}
\end{equation}
(remember that the weights $\omega_s^\sigma$ are all nonnegative and sum up to $1$). Plugged into \eqref{eq:def.PiK} this shows that
\begin{equation}\label{eq:bound.PiK}
\NORM{L^\infty(K)}{\Pi^K\bw_K}\lsim |\bw_K|_{\infty,K}.
\end{equation}
The bound \eqref{eq:bound.normK} follows by using \eqref{eq:bound.nablaK} and the Cauchy--Schwarz inequality to get $\NORM{L^2(K)}{\nabla^K\bw_D}\lsim h_K^{-1}|K|^{1/2}|\bw_K|_{\infty,K}$, and by using \eqref{eq:def.SK} and \eqref{eq:bound.PiK} in the definition \eqref{eq:def.normD} of $\NORM{1,K}{\bw_K}$, recalling that $h_K^{d-2}\lsim h_K^{-2}|K| $ and $\mathrm{Card}(\nodes_K)+\mathrm{Card}(\faces_K)\lsim 1$ by mesh regularity.
\end{proof}

\begin{proposition}[Stability of the averaged interpolator]\label{prop:stability.IDave}
It holds
\begin{equation}\label{eq:stability.interpolator}
\NORM{1,\D}{\IUD^{i,\ave}\bv}\lsim \NORM{L^2(\Omega)}{\nabla \bv}\qquad
\forall \bv\in H^1_0(\Omega\backslash \Gamma_i)^d.
\end{equation}
\end{proposition}

\begin{proof}

Let $\bv\in H^1_0(\Omega\backslash \Gamma_i)^d$.
The proof is split in several steps. The idea is to first obtain a bound on a cell $K$, by locally approximating $\bv$ by a linear function $\bq$ and using the linear exactness of $\IUD^{i,\ave}$ resulting from the properties of $\varpi_\Ks$. This allows us to bound the interpolates of $\bq$ and of $\bv-\bq$, and to conclude.

\emph{Step 1: local linear approximation $\bq$}.

Let $K\in\cells$ and set 
$$
 \mathcal N(K)=\bigcup_{s\in\nodes_K}\bigcup_{L\in\Ksi}L
 $$
the patch around $K$ made of all the cells in $\Ksi$ for each $s$ vertex of $K$. By mesh regularity, it can be checked that $\mathcal N(K)$ is connected by star-shaped sets as per \cite[Definition 1.4]{hho-book}, and by definition of $\Ksi$ we have $\bv\in H^1(\mathcal N(K))^d$. As a consequence, the $L^2(\mathcal N(K))^d$-projection $\bq$ of $\bv$ on $\mathbb{P}^1(\mathcal N(K))^d$ enjoys the following approximation and stability properties:
\begin{align}
\label{eq:q.approx}
\NORM{L^2(\mathcal N(K))}{\bq-\bv}\lsim{}& h_K\NORM{L^2(\mathcal N(K))}{\nabla\bv},\\
\label{eq:q.approx.trace}
\NORM{L^2(\sigma)}{\bq-\bv}\lsim{}& h_K^{1/2}\NORM{L^2(\mathcal N(K))}{\nabla\bv}\quad\forall\sigma\in\faces_K,\\
\label{eq:q.norm}
\NORM{L^2(\mathcal N(K))}{\nabla\bq}\lsim{}& \NORM{L^2(\mathcal N(K))}{\nabla\bv}.
\end{align}
The relation \eqref{eq:q.approx} comes from \cite[Theorem 1.45, Eq.~(1.74)]{hho-book}, accounting for the fact that the diameter of $\mathcal N(K)$ is $\lsim h_K$ by mesh regularity; the trace approximation property \eqref{eq:q.approx.trace} can be established following the same arguments as for Eq.~(1.75) in this reference (which establishes the trace approximation property for $\sigma$ on the boundary of $\mathcal N(K)$); the boundedness \eqref{eq:q.norm} follows from \cite[Remark 1.47]{hho-book}.

\medskip

\emph{Step 2: bound on $\IUD\bq$}.

In the following, by abuse of notation we denote by $\IUD\bq$ the (non-averaged) local interpolate of $\bq$ on $K$, that is, the vector obtained gathering only the degrees of freedom \eqref{eq:def.ID} attached to the chosen cell $K$ (for $s\in\nodes_K$ and $\sigma\in\faces_{\Gamma,K}^+$). We also note that the DOFs lying on the boundary of $\Omega$ do not necessarily vanish for this local interpolate.

For all $s\in\nodes_K$ we have $(\IUD\bq)_{\Ks}=\bq(\mathbf{x}_s)$ and thus, by choice \eqref{eq:choice.weights.sigma} of the face weights and since $\bq$ is linear, it holds  $\nabla^{\Ksig}\IUD\bq=\nabla_\sigma\bq$ for all $\sigma\in\faces_K$ ($\nabla_\sigma$ being the tangential gradient), see \cite[Lemma 14.8]{gdm} for details. From \eqref{def:def.operators.sigma}, we easily infer 
\begin{equation}\label{eq:exact.linear.Ksig}
\Pi^\Ksig\IUD\bq=\bq_{|\sigma}.
\end{equation}
Plugging this into the second relation of \eqref{eq:def.ID} yields 
\begin{equation}\label{eq:IDq.Ksig}
(\IUD\bq)_\Ksig=0\quad\forall\sigma\in\faces_{\Gamma,K}^+.
\end{equation}

Recalling \eqref{eq:def.nablaK} and using again \cite[Lemma 14.8]{gdm} and the fact that $\bq$ is linear, we deduce that
\begin{equation}\label{eq:nablaK.ID.q}
  \nabla^K\IUD\bq=\nabla \bq.
\end{equation}
The definition \eqref{eq:def.PiK} of $\Pi^K$ then shows that $\Pi^K\IUD\bq=\bq$. Together with \eqref{eq:IDq.Ksig} and the definition \eqref{eq:def.SK} of $S_K$, this gives 
\begin{equation}\label{eq:linear.exact.SK}
S_K(\IUD\bq,\bw_\D)=0\quad\forall\bw_\D\in\UD.
\end{equation}
Finally, recalling \eqref{eq:nablaK.ID.q} and \eqref{eq:q.norm}, as well as the definition \eqref{eq:def.normD} of $\NORM{1,K}{{\cdot}}$, we obtain 
\begin{equation}\label{eq:norm.ID.q}
  \NORM{1,K}{\IUD\bq}\lsim \NORM{L^2(\mathcal N(K))}{\nabla\bv}.
\end{equation}

\medskip

\emph{Step 3: bound on $\IUD\bq-\IUD^{i,\ave}\bv$}.

We bound each type of degree of freedom (boundary vertex, internal vertex, fracture face), then use \eqref{eq:bound.normK} to conclude.

If $s\in\nodes_K^{\text{ext}}$ then $(\IUD^{i,\ave}\bv)_{\Ks}=0$ and $(\IUD\bq)_{\Ks}=\bq(\mathbf{x}_s)$. We can then take $\sigma\in\faces^{\text{ext}}$ that contains $s$ and write
\begin{align}
|(\IUD\bq)_{\Ks}-(\IUD^{i,\ave}\bv)_{\Ks}|\lsim|\sigma|^{-1/2}\NORM{L^2(\sigma)}{\bq}
\lsim{}& |\sigma|^{-1/2}h_K^{1/2}\NORM{L^2(\mathcal N(K))}{\nabla\bv}\nonumber\\
\lsim{}& |K|^{-1/2}h_K\NORM{L^2(\mathcal N(K))}{\nabla\bv},
\label{eq:bound.Ks.bdry}
\end{align}
where we have used the inverse Lebesgue inequality \cite[Lemma 1.25]{hho-book} for the first inequality, the bound \eqref{eq:q.approx.trace} together with $\bv_{|\sigma}=0$ (since $\sigma\subset\partial\Omega$) for the second inequality, and the mesh regularity assumption to conclude.

If $s\in\nodes_K^{\text{int}}$ then \eqref{eq:def.varpi.integral} and the fact that $\bq$ is linear ensure that 
$$
  (\IUD\bq)_{\Ks}=\bq(\mathbf{x}_s)=\frac{1}{|U_{\Ksi}|}\int_{U_\Ksi}\varpi_\Ksi\bq.
$$
Hence,
\begin{align}
|(\IUD\bq)_{\Ks}-(\IUD^{i,\ave}\bv)_{\Ks}|= \left|\frac{1}{|U_{\Ksi}|}\int_{U_\Ksi}\varpi_\Ksi(\bq-\mathbf{v})\right|
\lsim{}& |U_\Ksi|^{-1/2}\NORM{L^2(U_\Ksi)}{\bq-\mathbf{v}}\nonumber\\
\lsim{}& |K|^{-1/2} h_K \NORM{L^2(\mathcal N(K))}{\nabla\bv},
\label{eq:bound.Ks.int}
\end{align}
where the first inequality follows from the bound $|\varpi_\Ksi|\lsim 1$ and the Cauchy--Schwarz inequality, while the conclusion is obtained using $|U_{\Ksi}|\gtrsim |K|$ (see \eqref{eq:def.varpi.bounds}), $U_\Ksi\subset\mathcal N(K)$ and \eqref{eq:q.approx}.

Finally, we consider the face degrees of freedom. If $\sigma\in\faces_{\Gamma,K}^+\setminus \faces_{\Gamma_i,K}^{+}$ then by \eqref{eq:IDiave.sigma.notGammai} and \eqref{eq:IDq.Ksig} we have $(\IUD\bq)_{\Ksig}=(\IUD^{i,\ave}\bv)_{\Ksig}=0$. If $\sigma\in\faces_{\Gamma_i,K}^{+}$ then \eqref{eq:def.ID} and \eqref{eq:IDiave.sigma.Gammai} yield
\begin{equation}\label{eq:bound.Ksig.1}
(\IUD\bq)_{\Ksig}-(\IUD^{i,\ave}\bv)_{\Ksig}=\frac{1}{|\sigma|}\int_\sigma\gamma^{\Ksig}(\bq-\bv)-\Pi^\Ksig(\IUD\bq-\IUD^{i,\ave}\bv).
\end{equation}
By \eqref{eq:bound.Ks.bdry} and \eqref{eq:bound.Ks.int}, all vertex degrees of freedom of $\IUD\bq-\IUD^{i,\ave}\bv$ are $\mathcal O(|K|^{-1/2} h_K \NORM{L^2(\mathcal N(K))}{\nabla\bv})$. The operator $\Pi^\Ksig$ only depends on these degrees of freedom so, by the same arguments that led to \eqref{eq:bound.PiK}, we find that $\NORM{L^\infty(\sigma)}{\Pi^\Ksig(\IUD\bq_{|\sigma}-\IUD^{i,\ave}\bv)}\lsim |K|^{-1/2} h_K \NORM{L^2(\mathcal N(K))}{\nabla\bv}$. Plugging this into \eqref{eq:bound.Ksig.1} and using the Cauchy--Schwarz inequality, we infer that
\begin{align}
\left|(\IUD\bq)_{\Ksig}-(\IUD^{i,\ave}\bv)_{\Ksig}\right|\lsim{}&
  |\sigma|^{-1/2}\NORM{L^2(\sigma)}{\bq-\bv} + |K|^{-1/2} h_K \NORM{L^2(\mathcal N(K))}{\nabla\bv}\nonumber\\
  \lsim {}&|K|^{-1/2} h_K \NORM{L^2(\mathcal N(K))}{\nabla\bv},
\label{eq:bound.Ksig.2}
\end{align}
where the mesh regularity and \eqref{eq:q.approx.trace} were used to conclude.

Gathering \eqref{eq:bound.Ks.bdry}, \eqref{eq:bound.Ks.int} and \eqref{eq:bound.Ksig.2} and invoking the bound \eqref{eq:bound.normK} with $\bw_K$ gathering the degrees of freedom on $K$ of $\IUD\bq-\IUD^{i,\ave}\bv$ yields
\begin{equation}\label{eq:norm.IDq.IDavev}
\NORM{1,K}{\IUD\bq-\IUD^{i,\ave}\bv}\lsim \NORM{L^2(\mathcal N(K))}{\nabla\bv}.
\end{equation}

\medskip

\emph{Step 4: conclusion}.

Combining \eqref{eq:norm.ID.q} and \eqref{eq:norm.IDq.IDavev}, squaring, summing over $K\in\cells$ and gathering the integrals by cells we infer
$$
\NORM{1,\D}{\IUD^{i,\ave}\bv}^2\lsim \sum_{K\in\cells}\NORM{L^2(\mathcal N(K))}{\nabla\bv}^2
 =\sum_{K'\in\cells}\NORM{L^2(K')}{\nabla\bv}^2\mathrm{Card}\{K\in\cells\,:\,K'\in\mathcal N(K)\}.
$$
The proof of \eqref{eq:stability.interpolator} is complete by using the mesh regularity properties to see that $\mathrm{Card}\{K\in\cells\,:\,K'\in\mathcal N(K)\}\lsim 1$.
\end{proof}

\subsubsection{Proof of the inf-sup condition}

We are now ready to prove Theorem \ref{th:infsup}. Let $\la_\D\in \MD$. Recalling the definition \eqref{eq:def.H12.norm} of the $\JSdual$-like norm, the inf-sup condition \eqref{eq:inf-sup} holds provided we can prove that, for all $i\in I$ and all $\bv_i\in H^1_0(\Omega_i^+;\Gamma_i)^d$, there exists $\bv_\D\in\UDz$ such that
\begin{equation}\label{eq:inf-sup.vi}
\frac{\int_\Gamma \la_\D\cdot\jump{\bv_\D}_\D}{\NORM{1,\D}{\bv_\D}}\gtrsim 
  \frac{\int_{\Gamma_i} \la_\D \cdot\bv_i}{\NORM{H^1(\Omega_i^+)}{\bv_i}}.
\end{equation}

Since $\bv_i$ vanishes on $\partial\Omega_i^+\backslash\Gamma_i$, its extension $\widetilde{\bv}_i$ by $0$ to $\Omega$ belongs to $H^1_0(\Omega\backslash\Gamma_i)^d$ and satisfies $\NORM{H^1(\Omega\backslash\Gamma_i)}{\widetilde{\bv}_i}=\NORM{H^1_0(\Omega_i^+)}{\bv_i}$.
Let us consider $\bv_\D=\IUD^{i,\ave}\widetilde{\bv}_i$ with the interpolator defined by \eqref{eq:def.IDave}; by the construction done in Appendix \ref{sec:existence.varpi}, we can define this interpolator in such a way that, for all $s\in\overline{\Gamma_i}$ and $L$ on the negative side of $\Gamma_i$, $U_\Lsi$ is fully contained in $\Omega\backslash\Omega_i^+$ (since $U_\Lsi$ can be fully contained in any chosen cell that contains $s$). 
This ensures that $(\IUD^{i,\ave}\widetilde{\bv}_i)_\Lsi=0$ for all such $s$, and thus that $\Pi^\Lsig\IUD^{i,\ave}\widetilde{\bv}_i=0$ for all such $L$. Hence, by \eqref{eq:trace.interpolate} and since $\la_\D$ is piecewise constant on $\faces_\Gamma$,
$$
  \int_\Gamma \la_\D\cdot\jump{\bv_\D}_\D=
  \sum_{\sigma\in\faces_{\Gamma}}\la_\sigma\cdot\int_\sigma \jump{\IUD^{i,\ave}\widetilde{\bv}_i}_\sigma
  =
  \sum_{\sigma\in\faces_{\Gamma_i}}\la_\sigma\cdot\int_\sigma \gamma^\Ksig\widetilde{\bv}_i
  =
  \int_{\Gamma_i}\la_\D\cdot \bv_i,
$$
where the conclusion follows from the fact that $\gamma^\Ksig$ is the trace on the positive side, where $\widetilde{\bv}_i=\bv_i$.
Dividing throughout by $\NORM{1,\D}{\bv_\D}=\NORM{1,\D}{\IUD^{i,\ave}\widetilde{\bv}_i}\lsim \NORM{H^1(\Omega\backslash\Gamma_i)}{\widetilde{\bv}_i}=\NORM{H^1_0(\Omega_i^+)}{\bv_i}$ (see \eqref{eq:stability.interpolator}) we infer that \eqref{eq:inf-sup.vi} holds.

\begin{remark}[Jump-based norm on the fracture network]\label{rem:jump.norm}
An apparent simpler -- and perhaps more natural -- choice, that would prevent us from having to cut the fracture network into planar components $(\Gamma_i)_{i\in I}$, would be to define the $\JSdual$-norm using test functions $\bv\in H^1_0(\Omega\backslash\Gamma)^d$ and replacing their trace on the positive side of the fractures by their jump across the fractures. This choice is fully valid if we consider discretisations that have two bubble degrees of freedom, one on each side of the fracture, because in that case we have
$$
  \int_\sigma\jump{\IUD^{\ave}\bv}_\sigma=\int_\sigma\jump{\bv}\qquad\forall\bv\in H^1_0(\Omega\backslash\Gamma)^d.
$$
However, in the case of a single bubble degree of freedom located only on one side of the fracture (which has some practical interest in case of non-matching interfaces), the best relation that can be established between $\jump{\IUD^{\ave}\bv}_\sigma$ and $\bv$ is \eqref{eq:trace.interpolate} (dropping the index $i$ since we are discussing the option to handle the whole network at once), in which $\Pi^\Lsig\IUD^{\ave}\bv$ is not the trace of $\bv$ on the negative side, but some approximation thereof. As a consequence, when bubble degrees of freedom are only taken on the positive side of the network, we have to ensure that $\Pi^\Lsig\IUD^{\ave}\bv=0$ on the negative side; this is achieved in the proof above by working on each individual planar component of the network and using the zero extension outside $\Omega_i^+$ (which provides zero values on the negative side of $\Gamma$). For simple networks we might be able to create a $\bv$ that is zero on the whole negative side, but for complex network with multiple intersection lines, this does not look feasible, and explains why we resort to the decomposed norm \eqref{eq:def.H12.norm}.
\end{remark}


\subsection{Discrete Korn inequality}

\begin{theorem}[Discrete Korn inequality]\label{th:korn}
	It holds
	\begin{equation}\label{eq:Korn}
		\NORM{1,\D}{\bv_\D}^2\lsim \NORM{L^2(\Omega\backslash \overline{\Gamma})}{\bbeps_\D(\bv_\D)}^2 + S_\D(\bv_\D,\bv_\D)\qquad\forall \bv_\D\in\UDz.
	\end{equation}
\end{theorem}

\begin{proof}
  We begin by substituting the quantities $\grad^{\D}{\bv_{\D}}$ in $\NORM{1,\D}{\bv_\D}$ and $\bbeps_{\D} (\bv_{\D})$ with their equivalents $\grad_\cells (\Pi^{\D} \bv_{\D})$ and $\bbeps_\cells(\Pi^{\D} \bv_{\D})$ (where $\grad_\cells$ and $\bbeps_\cells$ are the broken gradient and strain on $\cells$) to expose the function $\Pi^{\D} \bv_{\D} \in \mathbb{P}^1(\cells)^d$, and enable the usage of \cite[Lemma 7.23]{hho-book} with adjustments to account for fractures. 
  
  Specifically, the node-averaging operator $\mathcal{I}_{\text{av},h}^1$ from \cite[Section 7.3.2]{hho-book} can be adapted to avoid crossing fractures, by defining the node-averaged value at a given $s$ on the side $K$ of the fracture as follows 
  $$
  (\mathcal{I}_{\text{av},h}^1 \Pi^\D \bv_\D)_\Ks = {1\over \# \Ks}\sum_{L\in \Ks} \Pi^L\bv_\D(\x_s). 
  $$  
  It results that the jumps $(\mathcal{I}_{\text{av},h}^1 \Pi^\D \bv_\D)_\Ks -\Pi^L\bv_\D(\x_s)$ can be controlled using only non-fracture faces, instead of all mesh faces as in \cite[Section 7.3.2]{hho-book}. Applying the techniques there, we obtain the following modified version of \cite[Eq.~(7.66)]{hho-book}:
	\begin{equation}
		\dsp \NORM{L^2(\Omega\backslash \overline{\Gamma})}{\grad_\cells (\Pi^{\D} \bv_{\D})}^2 \lsim \NORM{L^2(\Omega\backslash \overline{\Gamma})}{\bbeps_\cells(\Pi^{\D} \bv_{\D}) }^2 + \sum_{\sigma \in \faces \setminus \faces_{\Gamma}} h_{\sigma}^{-1} \NORM{L^2(\sigma)}{\jump{\Pi^{\D} \bv_{\D}}_{| \sigma}}^2. \label{ineq_HHO}
	\end{equation} 	
	To conclude, we need to bound the jump terms in the right-hand side. 	
	We consider a face $K|L=\sigma \in \faces^{\text{int}} \setminus \faces_{\Gamma}$. For $s  \in \nodes_{\sigma}$, noticing that $\bv_{\mathcal{K} s} = \bv_{\mathcal{L} s}$, we write
	\begin{align}
		| \jump{\Pi^{\D} \bv_{\D}}_{|\sigma}(\mathbf{x}_s) |^2 ={}&
		| (\Pi^{K} \bv_{\D})(\mathbf{x}_s) - (\Pi^{L} \bv_{\D})(\mathbf{x}_s) |^2\nonumber\\
		\le{}&  2|\bv_{\mathcal{K} s} - (\Pi^K \bv_{\D}) (\mathbf{x}_s)|^2 + 2|\bv_{\mathcal{L} s} - (\Pi^L \bv_{\D}) (\mathbf{x}_s)|^2 \nonumber\\
		\lsim {}& h_K^{2-d}(S_K(\bv_{\D},\bv_{\D}) + S_L(\bv_{\D},\bv_{\D})), \label{saut_S_ineq}
	\end{align}
	where the conclusion follows from the definition \eqref{eq:def.SK} of the local stabilisation terms and the mesh regularity (which ensures $h_K\approx h_L$). Further, the fact that $\jump{\Pi^{\D} \bv_{\D}}_{| \sigma} \in \mathbb{P}^1(\sigma)^d$ along with \eqref{saut_S_ineq} yields
	\begin{align}
		\NORM{L^2(\sigma)}{\jump{\Pi^{\D} \bv_{\D}}_{| \sigma}}^2 \le{}&
		|\sigma|\NORM{\infty}{\jump{\Pi^{\D} \bv_{\D}}_{| \sigma}}^2 \lsim h_K^{d-1} \max_{s \in \nodes_{\sigma}}  |\jump{\Pi^{\D} \bv_{\D}}_{| \sigma} (\mathbf{x}_s)|^2  \nonumber\\
		\lsim {}& h_K (S_K(\bv_{\D},\bv_{\D}) + S_L(\bv_{\D},\bv_{\D})).
		\label{eq:est.korn.jump}
	\end{align}
	In the second inequality, we have used the fact that, by mesh regularity property, any $\x\in\sigma$ can be written $\x=\sum_{s\in\nodes_\sigma}\rho_s \x_s$ with $\sum_{s\in\nodes_\sigma}\rho_s=1$ and $\sum_{s\in\nodes_\sigma}|\rho_s|\lsim 1$, and that $\jump{\Pi^{\D} \bv_{\D}}(\x)=\sum_{s\in\nodes_\sigma}\rho_s\jump{\Pi^{\D} \bv_{\D}}(\x_s)$.
	We deduce from \eqref{eq:est.korn.jump} that 
	\begin{equation}
		h_{\sigma}^{-1} \NORM{L^2(\sigma)}{\jump{\Pi^{\D} \bv_{\D}}_{| \sigma}}^2  \lsim  S_K(\bv_{\D},\bv_{\D}) + S_L(\bv_{\D},\bv_{\D}).\label{eqa_1} 
	\end{equation}
		
	Consider now a face $\sigma \in \faces^{\text{ext}} \cap \faces_K$, $K \in \cells$. For $s \in \nodes_{\sigma}$, noticing that $\bv_{\mathcal{K} s} = 0$, we obtain
	$$
		| \jump{\Pi^{\D} \bv_{\D}}_{| \sigma}(\mathbf{x}_s) |^2 =
		| (\Pi^{K} \bv_{\D})(\mathbf{x}_s) - \bv_{\mathcal{K} s}  |^2 \le  h_K^{2-d}S_K(\bv_{\D},\bv_{\D}).
	$$
  Then, following the same arguments that lead to \eqref{eqa_1}, we get
	\begin{equation}
		h_{\sigma}^{-1} \NORM{L^2(\sigma)}{\jump{\Pi^{\D} \bv_{\D}}_{| \sigma}}^2  \lsim  S_K(\bv_{\D},\bv_{\D}).\label{eqa_2} 
	\end{equation}
	
	Summing \eqref{eqa_1} over the faces $\sigma \in  \faces^{\text{int}} \setminus \faces_{\Gamma}$ and \eqref{eqa_2} over the faces  $\sigma \in \faces^{\text{ext}}$,  we deduce
	\begin{equation}
		\sum_{\sigma \in  \faces \setminus \faces_{\Gamma}} 	h_{\sigma}^{-1} \NORM{L^2(\sigma)}{\jump{\Pi^{\D} \bv_{\D}}_{| \sigma}}^2  \lsim  \sum_{K \in \cells} S_K(\bv_{\D},\bv_{\D}).\label{ineq_saut_Pi_L2} 
	\end{equation}
	
	The proof of \eqref{eq:Korn} is concluded by noticing that
	$$ \NORM{1,\D}{\bv_\D}^2 =  \NORM{L^2(\Omega\backslash \overline{\Gamma})}{\grad_\cells (\Pi^{\D} \bv_{\D})}^2 +  \sum_{K \in \cells} S_K(\bv_{\D},\bv_{\D}),$$
 and by using \eqref{ineq_HHO} and \eqref{ineq_saut_Pi_L2}.	
\end{proof}

\subsection{Proof of the abstract error estimate}\label{sec:proof.abstract.error}

We prove here Theorem \ref{abstr.err.est}.
To shorten the notations, let us define the discrete energy inner product $\IPbracket{\cdot}{\cdot}_{e,\D}$ such that, for $\bu_{\D},\bv_{\D} \in \UD$,
\begin{equation*}
\IPbracket{\bu_{\D}}{\bv_{\D}}_{e,\D} = \int_{\Omega} \bbsigma_{\D} (\bu_{\D}): \bbeps_{\D}\left(\bv_{\D} \right)  + S_{\mu,\lambda,\D}\left( \bu_{\D}, \bv_{\D} \right), 	
\end{equation*}	
and denote by $\NORM{e,\D}{{\cdot}}$ its associated norm. From the discrete Korn inequality \eqref{eq:Korn}, we deduce the following bound for all $\bv_{\D} \in \UDz$:
\begin{equation}\label{relat.norme.discrete}
\NORM{1,\D}{\bv_{\D}} \lsim \NORM{e,\D}{\bv_{\D}}.
\end{equation}
  
Since $\bu$ is a weak solution to the contact problem and by the regularity assumed on $\la$ in the theorem, we have $-\div(\bbsigma(\bu))=\mathbf{f}\in L^2(\Omega)^d$ and $\la = -\gamma_\normal^+\bbsigma(\bu)\in L^2(\Gamma)^d$. Hence, $\bbsigma(\bu)\in\mathbf{W}$ defined by \eqref{def:esp.strain} and, using the definition \eqref{def:WD} of $w_\D$, we have for all $\bw_\D\in\UDz$
\begin{equation}\label{cons.dual_equ_1}
\int_{\Omega} \bbsigma(\bu): \bbeps_{\D}(\bw_{\D} ) -  \int_{\Omega}  \mathbf{f} \cdot \wt{\Pi}^{\D}\bw_{\D} =  -w_{\D}(\bbsigma(\bu),\bw_{\D}) - \int_{\Gamma} \la \cdot \jump{\bw_{\D}}_{\D}.
\end{equation}	 
Subtracting \eqref{eq:meca.var.tresca} (with $\bv_\D=\bw_\D$) from \eqref{cons.dual_equ_1}, we obtain
\begin{equation}\label{cons.dual.somme.var.formul}
	\int_{\Omega} (\bbsigma(\bu)-\bbsigma_{\D}(\bu_{\D})): \bbeps_{\D}(\bw_{\D} ) - S_{\mu,\lambda,\D}\left( \bu_{\D}, \bw_{\D} \right)   =  - w_{\D}(\bbsigma(\bu),\bw_{\D}) + \int_{\Gamma} (\la_{\D} - \la) \cdot \jump{\bw_{\D}}_{\D}.
\end{equation}	 
Take $\bv_{\D} \in \UDz$ and set $\bw_{\D} = \bv_{\D} - \bu_{\D}$ in \eqref{cons.dual.somme.var.formul} to get
\begin{align}\label{estima.1}
		 \NORM{e,\D}{\bv_{\D} - \bu_{\D}}^2    =  {}& -w_{\D}(\bbsigma(\bu),\bv_{\D}-\bu_{\D}) -\int_{\Omega} (\bbsigma(\bu)-\bbsigma_{\D}(\bv_{\D})): \bbeps_{\D}(\bv_{\D}-\bu_{\D})  \nonumber  \\
		 &+ S_{\mu,\lambda,\D}\left(\bv_{\D},\bv_{\D}-\bu_{\D} \right) 		 + \int_{\Gamma} (\la_{\D} - \la) \cdot ( \jump{\bv_{\D}}_{\D} - \jump{\bu}) \nonumber\\
		 &+  \int_{\Gamma} (\la_{\D} - \la) \cdot (\jump{\bu} - \jump{\bu_{\D}}_{\D} ).  
	\end{align}
Recalling that $\la_\D \in \CD\subset\bm{C}_f$, we deduce from \eqref{Lagrange_meca_contactfriction_2}  that $\int_{\Gamma}(\la_{\D} - \la)\cdot  \jump{\bu} \le 0$. Furthermore, as $\pi^{0}_{\faces_{\Gamma}} \la \in \CD$ (since $\trescacof$ is assumed to be piecewise constant on $\faces_\Gamma$), we obtain from the fact that $\jump{\bu_\D}_\D$ is piecewise constant on $\faces_\Gamma$ and  \eqref{eq:meca.ineq.contact.tresca} that 
$$  \int_{\Gamma}(\la - \la_{\D} )\cdot  \jump{\bu_{\D}}_{\D}  = \int_{\Gamma}(\pi^{0}_{\faces_{\Gamma}}\la  - \la_{\D} )\cdot  \jump{\bu_{\D}}_{\D} \le 0,$$
and consequently
$$
	\int_{\Gamma} (\la_{\D} - \la) \cdot (\jump{\bu} - \jump{\bu_{\D}}_{\D} ) \le 0.
$$
Plugging this relation into \eqref{estima.1}, using the norm estimate \eqref{relat.norme.discrete}, the definitions \eqref{def:CD} of $C_\D$ and \eqref{def:WD} of $\mathcal W_\D$, a Cauchy--Schwarz inequality on $S_{\mu,\lambda,\D}$ and Young's inequality,  we infer
	\begin{equation}\label{estima_1}
		\NORM{1,\D}{\bv_{\D} - \bu_{\D}}^2 \lsim  \mathcal{W}_{\D}(\bbsigma(\bu))^2 + C_{\D}(\bu,\bv_{\D})^2 +  \int_{\Gamma} (\la_{\D} - \la) \cdot ( \jump{\bv_{\D}}_{\D} - \jump{\bu}).
	\end{equation}
	
	We now return to \eqref{cons.dual.somme.var.formul}, which shows that, for all $\bw_{\D},\bv_{\D} \in \UDz$,
\begin{align}
\int_{\Gamma} (\la_{\D} - \la) \cdot \jump{\bw_{\D}}_{\D} = {}&  w_{\D}(\bbsigma(\bu),\bw_{\D}) +	\int_{\Omega} (\bbsigma(\bu)-\bbsigma_{\D}(\bu_{\D})): \bbeps_{\D}(\bw_{\D} )\nonumber\\
  &- S_{\mu,\lambda,\D}( \bv_{\D}, \bw_{\D}) - S_{\mu,\lambda,\D}( \bu_{\D} - \bv_{\D}, \bw_{\D}). \label{lamd.lam.w}
\end{align} 
Setting $\m_{\D} = \ProjFG \la$ and noticing that  $\int_{\Gamma} \la \cdot \jump{\bw_{\D}}_{\D} = \int_{\Gamma} \bm{\mu}_\D \cdot \jump{\bw_{\D}}_{\D}$ since $\jump{\bw_\D}_\D$ is piecewise constant on $\faces_\Gamma$, we deduce from \eqref{lamd.lam.w} that 
\begin{align*}
	\int_{\Gamma} (\la_{\D} - \m_{\D}) \cdot \jump{\bw_{\D}}_{\D} = {}&  w_{\D}(\bbsigma(\bu),\bw_{\D}) +	\int_{\Omega} (\bbsigma(\bu)-\bbsigma_{\D}(\bu_{\D})): \bbeps_{\D}(\bw_{\D} )\\
	&- S_{\mu,\lambda,\D}( \bv_{\D}, \bw_{\D}) - S_{\mu,\lambda,\D}( \bu_{\D} - \bv_{\D}, \bw_{\D}). \nonumber
\end{align*}
The discrete inf-sup condition \eqref{eq:inf-sup} then leads to
\begin{align*}
\NORM{-\nicefrac12,\Gamma}{\la_\D - \m_{\D}} \lsim & \sup_{\bw_\D\in\UDz}\frac{\int_\Gamma (\la_\D - \mu_{\D})\cdot\jump{\bw_\D}_\D}{\NORM{1,\D}{\bw_\D}} \\
\lsim {}&  \mathcal{W}_{\D}(\bbsigma(\bu)) + \NORM{L^2(\Omega\backslash\ov\Gamma)}{\bbsigma(\bu)-\bbsigma_{\D}(\bv_{\D})} + \NORM{L^2(\Omega\backslash\ov\Gamma)}{\bbsigma_{\D}(\bu_{\D} - \bv_{\D})}\\
  & + S_{\mu,\lambda,\D}( \bv_{\D},  \bv_{\D})^{1/2} +  S_{\mu,\lambda,\D}( \bu_{\D} - \bv_{\D}, \bu_{\D} - \bv_{\D})^{1/2},
\end{align*} 
from which we infer
\begin{equation}\label{lamd.lam.w.2}
	\NORM{-\nicefrac12,\Gamma}{\la_\D - \m_{\D}} \lsim  \NORM{1,\D}{\bu_{\D} - \bv_{\D}} +
	 \, \mathcal{W}_{\D}(\bbsigma(\bu)) + C_{\D}(\bu,\bv_{\D}). 
\end{equation} 
Combining \eqref{estima_1} and \eqref{lamd.lam.w.2} yields
\begin{align}
\NORM{1,\D}{\bu_\D - \bv_{\D}}^2 +	\NORM{-\nicefrac12,\Gamma}{\la_\D - \m_{\D}}^2 \lsim {}& \mathcal{W}_{\D}(\bbsigma(\bu))^2 + C_{\D}(\bu,\bv_{\D})^2\nonumber\\
& + \int_{\Gamma} (\la_{\D} - \la) \cdot ( \jump{\bv_{\D}}_{\D} - \jump{\bu}).
\label{eq:est.error.a}
\end{align}
Let us introduce the discrete $\Hhalfdual$-norm, dual of the discrete $\Hhalf$-norm \eqref{def:norm.1/2D} and  defined for all $\m \in L^2(\Gamma)^d$ by
\begin{equation}\label{def:norm.m1/2D}
	\NORM{-\nicefrac 12,\D}{\m} = \(\sum_{\sigma \in \faces_{\Gamma}} h_{\sigma} \NORM{L^2(\sigma)}{\m}^2\)^{1/2}. 
\end{equation}
Using a weighted Cauchy--Schwarz inequality, we note that 
\begin{equation}\label{eq:CS.Hhalf.discrete}
\int_\Gamma \bm{\mu}\cdot\bm{\xi}\le \NORM{-\nicefrac 12,\D}{\bm{\mu}}\NORM{\nicefrac 12,\D}{\bm{\xi}}\quad\forall \bm{\mu},\bm{\xi}\in L^2(\Gamma)^d.
\end{equation}
Then, the contact term in the right-hand side of \eqref{eq:est.error.a} can be estimated by introducing $\bm{\mu}_\D=\ProjFG\la$ and writing
\begin{align*}
 \int_{\Gamma} (\la_{\D} - \la) \cdot ( \jump{\bv_{\D}}_{\D} - \jump{\bu}) =  {}&  \int_{\Gamma} (\la_{\D} - \m_{\D}) \cdot ( \jump{\bv_{\D}}_{\D} - \ProjFG\jump{\bu}) +  \int_{\Gamma} (\m_{\D} - \la) \cdot ( \jump{\bv_{\D}}_{\D} - \jump{\bu}) \\
 \le {}&\frac{c}{2}\NORM{-\nicefrac 12,\D}{\la_{\D} - \m_{\D}}^2 + \frac{1}{2c} \NORM{\nicefrac 12,\D}{\jump{\bv_{\D}}_{\D} - \ProjFG\jump{\bu}}^2 \\
 & +\NORM{L^2(\Gamma)}{\mu_{\D} - \la}\, \NORM{L^2(\Gamma)}{\jump{\bv_{\D}}_{\D} - \jump{\bu}},
\end{align*}
where the introduction of the projector $\ProjFG$ in front of $\jump{\bu}$ in the first line is justified by the fact that $\la_{\D} - \m_{\D}$ is piecewise constant on $\faces_\Gamma$, and the conclusion follows from Cauchy--Schwarz inequalities (including \eqref{eq:CS.Hhalf.discrete}) and a Young inequality (with an arbitrary $c>0$).
We then apply Lemma \ref{lemma.disc.cont.norm1/2} in the appendix to write $\NORM{-\nicefrac 12,\D}{\la_{\D} - \m_{\D}}^2\lsim \NORM{-\nicefrac12,\Gamma}{\la_{\D} - \m_{\D}}^2$ and plug the resulting inequality in \eqref{eq:est.error.a}. Adjusting $c$ to absorb $\NORM{-\nicefrac12,\Gamma}{\la_{\D} - \m_{\D}}^2$ in the left-hand side, we infer
\begin{equation}\label{est_square}
\begin{aligned}  
\NORM{1,\D}{\bu_\D - \bv_{\D}}^2 +{}&	\NORM{-\nicefrac12,\Gamma}{\la_\D - \m_{\D}}^2 \lsim  \NORM{L^2(\Gamma)}{ \m_{\D}-\la}\NORM{L^2(\Gamma)}{\jump{\bv_{\D}}_{\D} - \jump{\bu}} \\
& + \NORM{\nicefrac 12,\D}{\jump{\bv_{\D}}_{\D} - \ProjFG\jump{\bu}}^2  +  \mathcal{W}_{\D}(\bbsigma(\bu))^2 + C_{\D}(\bu,\bv_{\D})^2.
\end{aligned}
\end{equation}
Applying triangular inequalities we can write
\begin{align}
   {}&\NORM{L^2(\Omega\backslash\ov\Gamma)}{\grad^{\D}\bu_\D - \grad \bu}+	\NORM{-\nicefrac 12,\Gamma}{\la_\D - \la} \nonumber\\
   {}&\lsim\NORM{L^2(\Omega\backslash\ov\Gamma)}{\grad^{\D}\bu_\D - \grad^\D \bv_\D}   + \NORM{-\nicefrac12,\Gamma}{\la_\D - \m_{\D}}  + \NORM{L^2(\Omega\backslash\ov\Gamma)}{\grad^{\D}\bv_\D - \grad \bu} + \NORM{-\nicefrac 12,\Gamma}{\m_\D - \la}\nonumber\\
   {}&\lsim  \NORM{1,\D}{\bu_\D - \bv_{\D}} + \NORM{-\nicefrac12,\Gamma}{\la_\D - \m_{\D}}  + C_{\D}(\bu,\bv_{\D}) + \NORM{-\nicefrac 12,\Gamma}{\m_\D - \la}.   
\label{eq:abstract.est.fin}
\end{align}
Taking the square root of \eqref{est_square}, plugging the resulting  estimate into \eqref{eq:abstract.est.fin}, recalling that $\bm{\mu}_\D=\ProjFG\la$, then taking the infimum over $\bv_\D$ concludes the proof of \eqref{estimat_error}.

\subsection{Proof of the error estimate}\label{sec:proof.error}

Theorem \ref{error_estimate} directly follows from the abstract error estimate \eqref{estimat_error}, Lemma \ref{lemma:approx.proj.Gamma} (with $s=1$) in the appendix, and Lemmas \ref{lem:consistency.nabla}, \ref{lem:consistency.jump} and \ref{lem:adjoint.consistency} below.

\begin{lemma}[Consistency of the gradient reconstruction]\label{lem:consistency.nabla}
If $\bu\in \U_0\cap H^2(\cells)^d$ then, recalling the definition \eqref{def:CD} of $C_\D$, it holds
\begin{equation}\label{eq:interp.u.H2}
C_\D(\bu,\IUD\bu)\lesssim h_\D\SEMINORM{H^2(\cells)}{\bu}.
\end{equation}
\end{lemma}

\begin{proof}
We first note that the regularity assumption ensures that $\bu\in C^0_0(\overline{\Omega}\backslash\Gamma)$: the $H^2$ regularity ensures the continuity of $\bu$ on each $\overline{K}$, and the values on each side of each $\sigma\not\in\faces_\Gamma$ match by the $H^1$ regularity. Hence, $\IUD\bu$ is well-defined.

Let $K\in\cells$ and $\q$ be the $L^2(K)$-orthogonal projection of $\bu$ on $\Poly{1}(K)^d$. By the approximation properties of the polynomial projector \cite[Theorem 1.45]{hho-book}, we have
\begin{equation}\label{eq:approx.poly}
\SEMINORM{H^s(K)}{\bu-\q}\lesssim h_K^{2-s}\SEMINORM{H^2(K)}{\bu}\,,\quad \forall s\in\{0,1,2\}.
\end{equation}
Applying the bound \cite[Eq.~(5.110)]{hho-book} to $\bu-\q$ yields
\begin{align}
\max_{\overline{K}}|\bu-\q|\lesssim{}& |K|^{-1/2}\(\NORM{L^2(K)}{\bu-\q}+h_K\SEMINORM{H^1(K)}{\bu-\q}+h_K^2\SEMINORM{H^2(K)}{\bu-\q}\)\nonumber\\
\lesssim{}& |K|^{-1/2}h_K^2\SEMINORM{H^2(K)}{\bu}.
\label{eq:approx.poly.sup}
\end{align}
where the conclusion follows from \eqref{eq:approx.poly}.
Plugging this estimate into the bound \eqref{eq:bound.normK} for $\bw_K$ gathering the DOFs on $K$ of $\IUD(\bu-\q)$ and recalling the definition \eqref{eq:def.normD} of the local norm, we infer
$$
\(\NORM{L^2(K)}{\nabla^K\IUD(\bu-\q)}^2+S_K(\IUD\bu-\IUD\q,\IUD\bu-\IUD\q)\)^{1/2}\lesssim h_K\SEMINORM{H^2(K)}{\bu}.
$$
The linear exactness properties \eqref{eq:nablaK.ID.q} and \eqref{eq:linear.exact.SK} allow us to manipulate the left-hand side to obtain
$$
\(\NORM{L^2(K)}{\nabla^K\IUD\bu-\nabla \q}^2+S_K(\IUD\bu,\IUD\bu)\)^{1/2}\lesssim h_K\SEMINORM{H^2(K)}{\bu}.
$$
Recalling the definition \eqref{def:CD} of $C_\D(\bu,\IUD\bu)$, the estimate \eqref{eq:interp.u.H2} follows by introducing $\pm\nabla \bu$ in the left-hand side above, using a triangle inequality, invoking again the approximation property \eqref{eq:approx.poly} (with $s=1$), squaring, summing over $K$, using the bound $h_K\le h_\D$ and taking the square root. \end{proof}

\begin{lemma}[Consistency of the jump reconstructions]\label{lem:consistency.jump}
If $\bu\in H^2(\cells)^d$ then
\begin{equation}\label{eq:approx.jump.L2}
\NORM{L^2(\Gamma)}{\jump{\IUD\bu}_{\D} - \jump{\bu}} \lsim h_\D \left(\SEMINORM{H^1(\faces_\Gamma)}{\jump{\bu}} 
+h_\D^{1/2}\SEMINORM{H^2(\cells)}{\bu}\right)
\end{equation}
and
\begin{equation}\label{eq:approx.jump.H12}
\NORM{1/2,\D}{\jump{\IUD\bu}_{\D} - \ProjFG\jump{\bu}} \lsim h_\D \SEMINORM{H^{2}(\cells)}{\bu}
\end{equation}
\end{lemma}

\begin{proof}
These approximation properties are, similarly to \eqref{eq:interp.u.H2}, a consequence of the linear exactness of the jump reconstruction (upon projection on piecewise constant functions) and of a bound on this operator.

The definition \eqref{eq:def.jump.sigma} of the jump reconstruction directly gives
\begin{equation}\label{eq:bound.jump}
|\jump{\bv_\D}_\sigma|\lesssim \max_{s\in\nodes_\sigma}|\bv_\Ks| +\max_{s\in\nodes_\sigma}|\bv_\Ls|+ |\bv_\Ksig|\quad
\forall \bv_\D\in\UD\,,\ \forall\sigma=K|L\in\faces_\Gamma.
\end{equation}
Consider now, for such a $\sigma$, two linear functions $\bq_K\in\Poly{1}(K)^d$, resp.~$\bq_L\in\Poly{1}(L)^d$, that satisfy \eqref{eq:approx.poly.sup} for $K$, resp.~for $L$. Let $\bq$ be the piecewise polynomial on $K\cup L$ defined by $\bq_K$ and $\bq_L$.
Applying \eqref{eq:approx.poly.sup} on each side of $\sigma$ we have
\begin{equation}\label{eq:approx.jump.uq}
\NORM{L^\infty(\sigma)}{\jump{\bu-\bq}}\lesssim |K|^{-1/2}h_K^2\SEMINORM{H^2(K)}{\bu}+|L|^{-1/2}h_L^2\SEMINORM{H^2(L)}{\bu}.
\end{equation}
Moreover, even though $\bq$ is not globally defined or in $\U_0$, its interpolate on the degrees of freedom attached to $\sigma$ is well defined. The definition \eqref{eq:def.ID} of $\IUD$ easily shows that the degrees of freedom of $\IUD(\bu-\bq)$ on $\sigma$ are bounded by the maximum of $|\bu-\bq|$ on $K\cup L$. Hence, by \eqref{eq:bound.jump} and \eqref{eq:approx.poly.sup},
\begin{equation}\label{eq:approx.jump.IUDuq}
|\jump{\IUD(\bu-\bq)}_\sigma|\lesssim  |K|^{-1/2}h_K^2\SEMINORM{H^2(K)}{\bu}
+|L|^{-1/2}h_L^2\SEMINORM{H^2(L)}{\bu}.
\end{equation}
The linear exactness \eqref{eq:exact.linear.Ksig} applied on each side of $\sigma$ together with \eqref{eq:IDq.Ksig} shows that $\jump{\IUD\bq}_\sigma=\pi^0_\sigma\jump{\bq}$ (with $\pi^0_\sigma$ the $L^2$-projector on $\Poly{0}(\sigma)^d$). We can therefore write
\begin{align}
\NORM{L^2(\sigma)}{\jump{\IUD\bu}_\sigma-\jump{\bu}}
\le{}&
  \NORM{L^2(\sigma)}{\jump{\IUD(\bu-\bq)}_\sigma}
  + \NORM{L^2(\sigma)}{\pi^0_\sigma\jump{\bq-\bu}}
  + \NORM{L^2(\sigma)}{\pi^0_\sigma\jump{\bu}-\jump{\bu}}\nonumber\\
\lesssim{}&
  |\sigma|^{1/2}\left(|K|^{-1/2}h_K^2\SEMINORM{H^2(K)}{\bu}
+|L|^{-1/2}h_L^2\SEMINORM{H^2(L)}{\bu}\right)\nonumber\\
  &+\NORM{L^2(\sigma)}{\pi^0_\sigma\jump{\bu}-\jump{\bu}}\nonumber\\
\lesssim{}&
  h_\D^{3/2}\left(\SEMINORM{H^2(K)}{\bu}+\SEMINORM{H^2(L)}{\bu}\right)+ h_\sigma \SEMINORM{H^1(\sigma)}{\jump{\bu}},
\label{eq:bound.jump.long}
\end{align}
where we have introduced $\pm\pi^0_\sigma\jump{\bq-\bu}=\pm(\jump{\IUD\bq}_\sigma-\pi^0_\sigma\jump{\bu})$ and used a triangle inequality in the first line, invoked \eqref{eq:approx.jump.IUDuq} and used the $L^2(\sigma)$-boundedness of $\pi^0_\sigma$ together with \eqref{eq:approx.jump.uq} in the second inequality, and used the mesh regularity property in the conclusion to write $|\sigma|^{1/2}h_K^{1/2}\lesssim |K|^{1/2}$ and $|\sigma|^{1/2}h_L^{1/2}\lesssim |L|^{1/2}$, together with the approximation properties of $\pi^0_\sigma$ \cite[Theorem 1.45]{hho-book}.
Squaring, summing over $\sigma$ and taking the square root yields \eqref{eq:approx.jump.L2}.

To prove the second estimate, we first notice that, by definition \eqref{eq:def.jump.sigma} of $\jump{{\cdot}}_\sigma$ and \eqref{eq:def.ID} of the interpolator,
$$
\jump{\IUD\bu}_\sigma=\frac{1}{|\sigma|}\int_\sigma (\trace^\Ksig\bu - \Pi^\Lsig(\IUD\bu))=\pi^0_\sigma(\trace^\Ksig\bu - \Pi^\Lsig(\IUD\bu)).
$$
Hence,
\begin{align*}
\NORM{L^2(\sigma)}{\jump{\IUD\bu}_{\sigma} - \pi^0_\sigma\jump{\bu}}
={}&\NORM{L^2(\sigma)}{\pi^0_\sigma(\trace^\Ksig\bu - \Pi^\Lsig(\IUD\bu) - \jump{\bu})}\\
\le{}& \NORM{L^2(\sigma)}{\trace^\Lsig\bu-\Pi^\Lsig(\IUD\bu)}.
\end{align*}
Introduce as above a linear approximation $\bq_L$ of $\bu$ in $L$, and use the linear exactness \eqref{eq:exact.linear.Ksig} (with $L$ instead of $K$) of $\Pi^\Lsig$ to deduce
\begin{equation*}
\NORM{L^2(\sigma)}{\jump{\IUD\bu}_{\sigma} - \pi^0_\sigma\jump{\bu}}
\le \NORM{L^2(\sigma)}{\bu_{|L}-\bq_L}+\NORM{L^2(\sigma)}{\Pi^\Lsig\IUD(\bu_{|L}-\bq_L)}.
\end{equation*}
In a similar way as \eqref{eq:bound.PiK} we can show that $\Pi^\Lsig\bw_\D$ is bounded above, on $\sigma$, by the maximum of the absolute values of the degrees of freedom associated with $\sigma$. By definition of $\IUD$ and \eqref{eq:approx.poly.sup} (with $L$ instead of $K$) and using similar arguments as for the first terms in \eqref{eq:bound.jump.long}, we infer
\begin{equation*}
\NORM{L^2(\sigma)}{\jump{\IUD\bu}_{\sigma} - \pi^0_\sigma\jump{\bu}}
\lesssim h_L^{3/2} \SEMINORM{H^2(L)}{\bu}.
\end{equation*}
Squaring, multiplying by $h_\sigma^{-1}\lesssim h_L^{-1}$, summing over $\sigma$ and taking the square root yields \eqref{eq:approx.jump.H12}.
\end{proof}

\begin{lemma}[Adjoint consistency] \label{lem:adjoint.consistency}
Recalling the definition \eqref{def:WD} of the adjoint consistency error $\mathcal{W}_\D$, we have  for all $\bbtau \in \mathbf{W}$ such that $\bbtau\in H^1(\cells)^{d\times d}$ the estimate 
        \begin{equation}\label{est_WD}
		\mathcal{W}_{\D}(\bbtau) \lsim  h_{\D} \SEMINORM{H^1(\cells)}{\bbtau}. 
	\end{equation} 
\end{lemma}

\begin{proof}
	For each $\bbtau \in \mathbf{W} \cap H^1(\cells)^{d\times d}$, the definitions \eqref{def:WD} of $w_\D$ and \eqref{eq:def.jump.sigma} of $\jump{{\cdot}}_\sigma$ yield 
	$$
		\begin{aligned}
			w_{\D}(\bbtau,\bv_{\D}) = &\sum_{K \in \cells}\(  - \(\int_{K} \bbtau \) : \bbeps_{K}(\bv_{\D} ) - \sum_{\sigma \in \faces_{K}} \overline{\bv}_K \cdot  \int_{\sigma} (\bbtau|_K\, \normal_{\Ksig}  ) \)\\
			&+  \sum_{\sigma = K|L \in \faces_{\Gamma}} \(\int_{\sigma} \bbtau|_K \, \normal_{K \sigma} \)\cdot(\overline{\bv}_{\Ksig} - \overline{\bv}_{\Lsig} + \bv_{\Ksig}  ) \\
			= &\sum_{K \in \cells}  \(-|K|\,\bbtau_K : \bbeps_{K}(\bv_{\D} ) - \sum_{\sigma \in \faces_{K}} |\sigma|\,\overline{\bv}_K \cdot  \boldsymbol{\tau}_{\Ksig} \)\\
			&+ \sum_{\sigma = K|L \in \faces_{\Gamma}} |\sigma| \boldsymbol{\tau}_{\Ksig} \cdot (\overline{\bv}_{\Ksig} - \overline{\bv}_{\Lsig} + \bv_{\Ksig}  ),\nonumber
		\end{aligned}
	$$	
	with
	$$
		\bbtau_K = \frac{1}{|K|} \int_K \bbtau \quad \text{and} \quad \boldsymbol{\tau}_{K\sigma} = \frac{1}{|\sigma|} \int_{\sigma} (\bbtau|_K \normal_{K \sigma}).
  $$
	Recalling the definition \eqref{eq:def.nablaK} of  $\grad^K$ and noticing that $\bbtau_K:\bbeps_K(\bv_{\D}) = \bbtau_K:\grad^K \bv_{\D}$ since $\bbtau_K$ is symmetric, we infer
	\begin{equation}\label{eq:adjoint.consistency.1}
		\begin{aligned}
			w_{\D}(\bbtau,\bv_{\D}) =& \sum_{K \in \cells} \sum_{\sigma \in \faces_K} \(-|\sigma|\, \overline{\bv}_{\Ksig} \cdot ( \bbtau_K \normal_{\Ksig}) -|\sigma|\, \overline{\bv}_{K} \cdot \boldsymbol{\tau}_{K\sigma} \) \\
			&-  \sum_{\sigma = K|L \in \faces_{\Gamma}} |\sigma|\, \bv_{\Ksig} \cdot ( \bbtau_K \normal_{\Ksig}) +  \sum_{\sigma = K|L \in \faces_{\Gamma}} |\sigma| \boldsymbol{\tau}_{K\sigma} \cdot (\overline{\bv}_{\Ksig} - \overline{\bv}_{\Lsig} + \bv_{\Ksig}  ).
		\end{aligned}
	\end{equation}
By the normal continuity property embedded in the space $\mathbf{W}$ we have $\boldsymbol{\tau}_{\Ksig}=-\boldsymbol{\tau}_{\Lsig}$ for all $\sigma=K|L\in\faces^{\text{int}}$. Since $\overline{\bv}_{\Ksig}=\overline{\bv}_{\Lsig}$ whenever $\sigma\not\in\faces_\Gamma$, we infer
  $$
 	\sum_{\sigma = K|L \in \faces_{\Gamma}} |\sigma|\, \boldsymbol{\tau}_{K\sigma} \cdot (\overline{\bv}_{\Ksig} -  \overline{\bv}_{\Lsig}) = \sum_{K \in \cells} \sum_{\sigma \in \faces_K} |\sigma|\, \overline{\bv}_{\Ksig} \cdot \boldsymbol{\tau}_{K\sigma}.
 	$$
	Moreover, as $\sum_{\sigma \in \faces_K} |\sigma|\,\normal_{\Ksig}=0$ for all $K\in\cells$,
	$$
		\sum_{K \in \cells} \sum_{\sigma \in \faces_K} \, |\sigma|\overline{\bv}_{K} \cdot ( \bbtau_{K} \normal_{\Ksig}) = \sum_{K \in \cells} \overline{\bv}_{K} \cdot \(\bbtau_{K} \sum_{\sigma \in \faces_K} |\sigma|\,\normal_{\Ksig}\)=0.
	$$
	Plugging these relations into \eqref{eq:adjoint.consistency.1} leads to
	$$
		\begin{aligned}
			w_{\D}(\bbtau,\bv_{\D}) ={}& \sum_{K \in \cells} \sum_{\sigma \in \faces_K} |\sigma|\, \overline{\bv}_{\Ksig} \cdot (-\bbtau_{K}\normal_{\Ksig} + \boldsymbol{\tau}_{K\sigma}) + \sum_{K \in \cells} \sum_{\sigma \in \faces_K} |\sigma|\, \overline{\bv}_{K} \cdot (\bbtau_{K}\normal_{\Ksig} - \boldsymbol{\tau}_{K\sigma})  \\
			&+  \sum_{\sigma = K|L \in \faces_{\Gamma}} |\sigma|\, \boldsymbol{\tau}_{K\sigma} \cdot \bv_{\Ksig} - \sum_{\sigma = K|L \in \faces_{\Gamma}} |\sigma|\, ( \bbtau_{K} \normal_{\Ksig}) \cdot \bv_{\Ksig}  \\
			={}& \sum_{K \in \cells} \sum_{\sigma \in \faces_K} |\sigma|\, (\overline{\bv}_{\Ksig} - \overline{\bv}_{K} ) \cdot (\boldsymbol{\tau}_{K\sigma} - \bbtau_{K}\normal_{\Ksig} )  +  \sum_{\sigma = K|L \in \faces_{\Gamma}} |\sigma|\, \bv_{K\sigma} \cdot (\boldsymbol{\tau}_{K\sigma} - \bbtau_{K}\normal_{\Ksig}).
		\end{aligned}
  $$	
  Using the Cauchy--Schwarz inequality and invoking Lemma \ref{inequ_pour_consist}, we infer
  \begin{align}
    w_{\D}(\bbtau,\bv_{\D}) \le{}& \left(\sum_{K \in \cells} \sum_{\sigma \in \faces_K} \frac{|\sigma|}{h_K} |\overline{\bv}_{\Ksig} - \overline{\bv}_{K} |^2\right)^{\nicefrac12} \left(\sum_{K \in \cells} \sum_{\sigma \in \faces_K} |\sigma|h_K|\boldsymbol{\tau}_{K\sigma} - \bbtau_{K}\normal_{\Ksig} |^2\right)^{\nicefrac12}\nonumber\\
    &+ \left( \sum_{\sigma = K|L \in \faces_{\Gamma}} \frac{|\sigma|}{h_K}|\bv_{K\sigma}| \right)^{\nicefrac12}
    \left( \sum_{\sigma = K|L \in \faces_{\Gamma}} |\sigma| h_K |\boldsymbol{\tau}_{K\sigma} - \bbtau_{K}\normal_{\Ksig}|^2\right)^{\nicefrac12}\nonumber\\
    \lesssim{}& \NORM{1,\D}{\bv_{\D}}  \left(\sum_{K \in \cells} \sum_{\sigma \in \faces_K} |\sigma| h_K |\boldsymbol{\tau}_{K\sigma} - \bbtau_{K}\normal_{\Ksig} |^2\right)^{\nicefrac12}.
  \label{est:limconf.1}
  \end{align}
  Note that the term $\sum_{\sigma = K|L \in \faces_{\Gamma}} |\sigma| h_K |\boldsymbol{\tau}_{K\sigma} - \bbtau_{K}\normal_{\Ksig}|^2$ has been included in the last factor in the right-hand side. By \cite[Lemma B.6]{gdm},
  $$
    |\sigma| h_K |\boldsymbol{\tau}_{K\sigma} - \bbtau_{K}\normal_{\Ksig} |^2\lesssim h_K^2\SEMINORM{H^1(K)}{\bbtau}^2.
  $$
  Plugging this into \eqref{est:limconf.1}, dividing by $\NORM{1,\D}{\bv_{\D}}$ and taking the supremum over $\bv_\D$ concludes the proof.
\end{proof}

\begin{lemma}\label{inequ_pour_consist}
	For all $\bv_{\D} \in \UDz$, the following two inequalities hold:
	\begin{align}
		\label{est:norm.v.Gamma}
		\(\sum_{\sigma =K|L \in \faces_{\Gamma}}  \frac{|\sigma|}{h_K}|\bv_{\Ksig}|^2 \)^{\nicefrac12}\lesssim{}& \NORM{1,\D}{\bv_{\D}} ,\\
    \label{est:norm.v.Omega}
		\(\sum_{K \in \cells} \sum_{\sigma \in \faces_K} \frac{|\sigma|}{h_K}|\overline{\bv}_{\Ksig} - \overline{\bv}_{K}|^2 \)^{\nicefrac12}\lesssim{}& \NORM{1,\D}{\bv_{\D}}.
	\end{align}
\end{lemma}

\begin{proof}	To prove \eqref{est:norm.v.Gamma}, we simply write, by mesh regularity property, $\frac{|\sigma|}{h_K}\lesssim h_K^{d-2}$, so that, by definition \eqref{eq:def.normD} of the discrete norm,
  $$
  \sum_{\sigma =K|L \in \faces_{\Gamma}}  \frac{|\sigma|}{h_K}|\bv_{\Ksig}|^2
  \lesssim \sum_{\sigma =K|L \in \faces_{\Gamma}}  h_K^{d-2}|\bv_{\Ksig}|^2 \lesssim S_\D(\bv_\D,\bv_\D)\lesssim  \NORM{1,\D}{\bv_{\D}}^2.
  $$
	
	We now turn to \eqref{est:norm.v.Omega}. Let $K\in\cells$ and $\sigma\in\faces_K$. By the choice \eqref{eq:choice.weights.sigma} of the weights $(\omega_{s}^\sigma)_{s\in\nodes_\sigma}$, the definition \eqref{def:def.operators.sigma} of $\overline{\bv}_{\Ksig}$ and the definition \eqref{eq:def.PiK} of the linear function $\Pi^K\bv_\D$, we have
	$$
	 \sum_{s \in \nodes_{\sigma}} \omega_{s}^\sigma(\bv_{\Ks} - \Pi^K \bv_{\D}(\mathbf{x}_s)) = \overline{\bv}_{\Ksig} - \overline{\bv}_{K}  - \grad^K\bv_{\D}(\ov\x_{\sigma} - \ov\x_{K}).
	 $$
	The convexity of the function $x \rightarrow |x|^2$ then yields	
	$$
	|\overline{\bv}_{\Ksig} - \overline{\bv}_{K}  - \grad^K\bv_{\D}(\ov\x_{\sigma} - \ov\x_{K})|^2 \le  \sum_{s \in \nodes_{\sigma}} \omega_{s}^\sigma|\bv_{\Ks} - \Pi^K \bv_{\D}(\mathbf{x}_s)|^2.
	$$
	Apply the inequality ${1 \over 2} |\boldsymbol{a}|^2 \leq |\boldsymbol{b}|^2 + |\boldsymbol{a}-\boldsymbol{b}|^2$ with $\boldsymbol{a} = \overline{\bv}_{\Ksig} - \overline{\bv}_{K} $ and $\boldsymbol{b} =  \grad^K\bv_{\D}(\ov\x_{\sigma} - \ov\x_{K})$. Recalling the definition \eqref{eq:def.SK} of $S_K$, and using the facts that each $\omega^s_\sigma$ is in $[0,1]$ and that the number of faces that meet at a vertex $s\in\nodes_K$ is $\lesssim 1$, we deduce that	
	\begin{equation*}
		\frac{1}{2}|\overline{\bv}_{\Ksig} - \overline{\bv}_{K}|^2 \le h_K^2\,|\grad^K \bu_{\D}|^2 + h_K^{2-d} S_K(\bv_{\D},\bv_{\D}).
	\end{equation*}
	The estimate \eqref{est:norm.v.Omega} then follows by multiplying by $|\sigma|/h_K$, by noticing that $|\sigma| h_K\lesssim |K|\lesssim h_K^d$, and by summing over $K\in\cells$.
\end{proof}

\section{Numerical Experiments} \label{sec:numerics}

\subsection{Unbounded 2D domain with a single fracture under compression}
\label{tchelepi_meca}
This test case presented in~\cite{contact-BEM,tchelepi-castelletto-2020,GKT16,droniou2023bubble} for a Coulomb frictional contact model also applies here to the Tresca frictional contact model \eqref{mode_tresca} since the normal traction of the analytical solution and the friction coefficient are constant along the fracture. It consists of a 2D unbounded domain containing a single fracture and subject to a compressive remote stress $\sigma$ = 100 MPa. The fracture inclination with respect to the $x$-direction is $\psi=\pi/9$ and its length is $2 \ell=2$ m.  The Coulomb-friction coefficient, Young's modulus and Poisson's ratio are set to  $F=1/\sqrt{3}$, $E=25$ GPa and $\nu=0.25$. The analytical solution in terms of the Lagrange multiplier $\lambda_{\n}$ and of the jump of the tangential displacement field is given by: 
\begin{equation}\label{sol.compression}
	\lambda_{\n}  = \sigma \sin^2(\psi),\quad | {\jump{\bu } }_{\tang}| = \frac{4(1-\nu)}{E}  \sigma \sin (\psi) \left( \cos(\psi) - \frac{\trescacof}{ \lambda_{\n}}\sin(\psi)\right)\sqrt{\ell^2 - (\ell^2 - \tau^2)},
\end{equation}
where $0\le \tau\le2\ell$ is a curvilinear abscissa along the fracture. Note that, since $\lambda_{\n}>0$, we have $\jump{\bu}_{\n}=0$ on the fracture. It results that the Tresca threshold is constant along the fracture and defined by $\trescacof = F\lambda_{\n}$.

Boundary conditions are imposed on $\bu$ at specific nodes of the mesh, as shown in Figure~\ref{test_compression}, to respect the symmetry of the expected solution.
For this simulation, we sample a $320 \text{m}\times320\,\rm m$ square, and carry out uniform refinements at each step in such a way to compute the solution on meshes containing 100, 200, 400, and 800 faces on the fracture (corresponding, respectively, to 12\,468, 49\,872, 199\,488, and 797\,952 triangular elements). The initial mesh is refined in a neighborhood of the fracture; starting from this mesh, we perform global uniform refinements at each step. 

Figure~\ref{compression_comparison} shows the comparison between the analytical and numerical Lagrange multipliers $\lambda_{\n}$  and tangential displacement jump $\jump{\bu }_{\tang}$ computed on the finest mesh. The Lagrange multiplier $\lambda_{\n}$ presents some oscillations in a neighborhood of the fracture tips. As already explained in~\cite{tchelepi-castelletto-2020}, this is due to the sliding of faces close to the fracture tips (in this test case, all fracture faces are in a contact-slip state). The discrete tangential displacement jump cannot be distinguished from the analytical solution on this fine mesh. 
Figure \ref{compression_error}  displays the convergence of the tangential displacement jump and of the normal Lagrange multiplier as a function of the size of the largest fracture face denoted by $h$. Note that the $L^2$ error for the Lagrange multiplier is computed $5\%$ away from each tip to circumvent the lack of convergence induced by the oscillations as in \cite{tchelepi-castelletto-2020}. A first-order convergence for the displacement jump and a $1.5$ convergence order for the Lagrange multiplier are observed. The former (low) rate is related to the low regularity of $\jump{\bu}_\tang$ close to the tips (cf.~the analytical expression~\eqref{sol.compression}), the latter (higher than expected) rate is likely related to the fact that $\lambda_{\n}$ is constant. 
\begin{figure}[H]
	\centering
	
	\subfloat[]{
		\raisebox{.45cm}{
			\includegraphics[keepaspectratio=true,scale=.35]{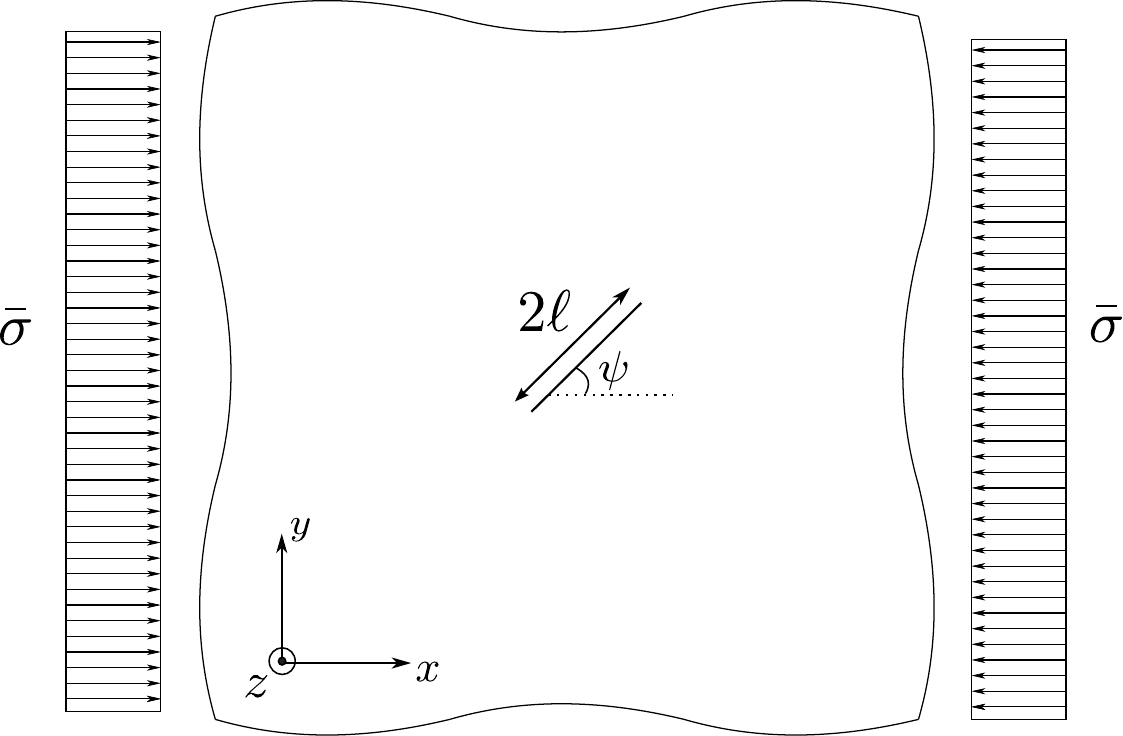}}
	}
	\hspace{2cm}
	\subfloat[]{\includegraphics[keepaspectratio=true,scale=.1]{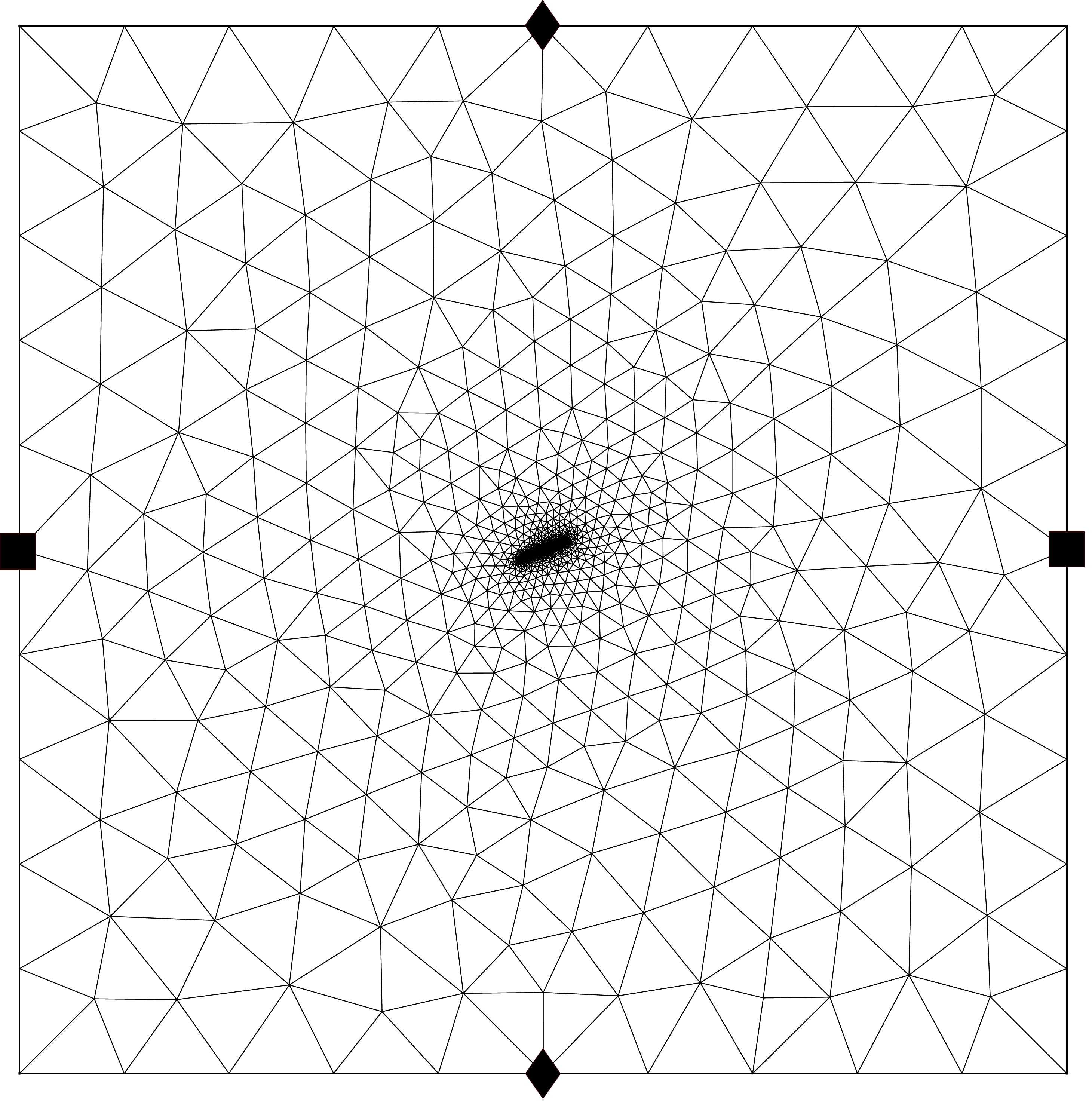}}\\
	\caption{Unbounded domain containing a single fracture under uniform compression (a) and mesh including nodes for boundary conditions ($\blacklozenge$: $u_x = 0$, $\blacksquare$: $u_y=0$), for the example of Section~\ref{tchelepi_meca}.}
	\label{test_compression}
\end{figure}
\begin{figure}[H]
	\centering
	\begin{tikzpicture}
		\node (img)  {\includegraphics[scale=0.55]{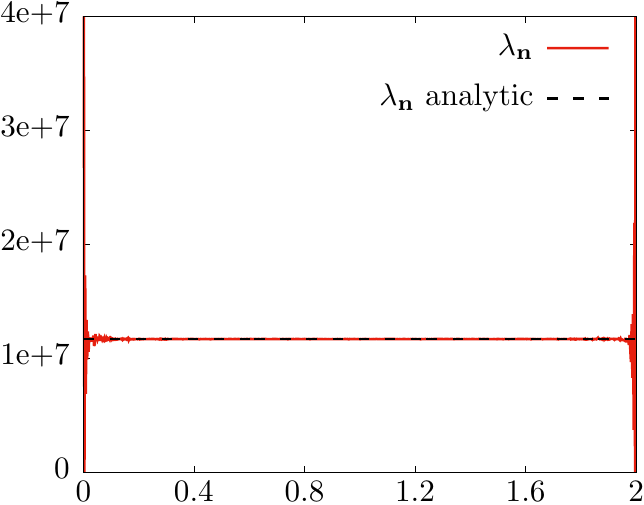}};
		\node[below=of img, node distance=0cm, rotate=0, anchor=center,yshift=0.8cm] {~~~~\footnotesize{ $\tau$ (m)}};
	\end{tikzpicture}
	\hspace{1cm}
	\begin{tikzpicture}
		\node (img)  {\includegraphics[scale=0.55]{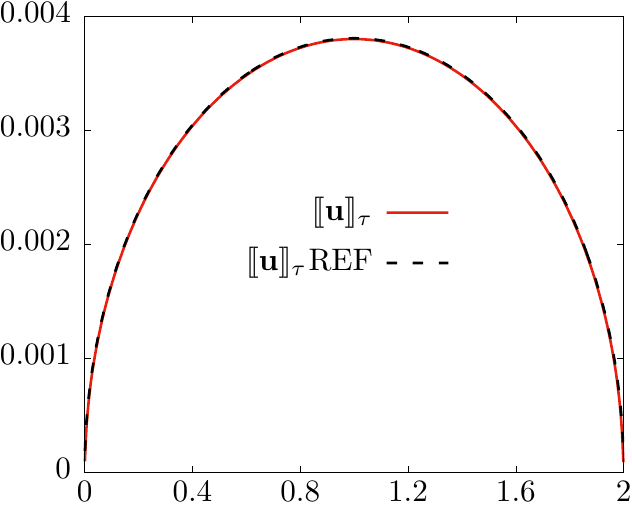}};
		\node[below=of img, node distance=0cm, rotate=0, anchor=center,yshift=0.8cm] {~~~~\footnotesize{ $\tau$ (m)}};
	\end{tikzpicture}
	
	\caption{Comparison between the numerical and analytical solutions on the finest mesh (with 800 fracture faces), in terms of $\lambda_{\n}$ (a) and $\jump{\bu }_{\tang}$ (b), example of Section~\ref{tchelepi_meca}.}
	\label{compression_comparison}
\end{figure}
\begin{figure}[H]
	\centering
	\begin{tikzpicture}
		\node (img)  {\includegraphics[scale=0.6]{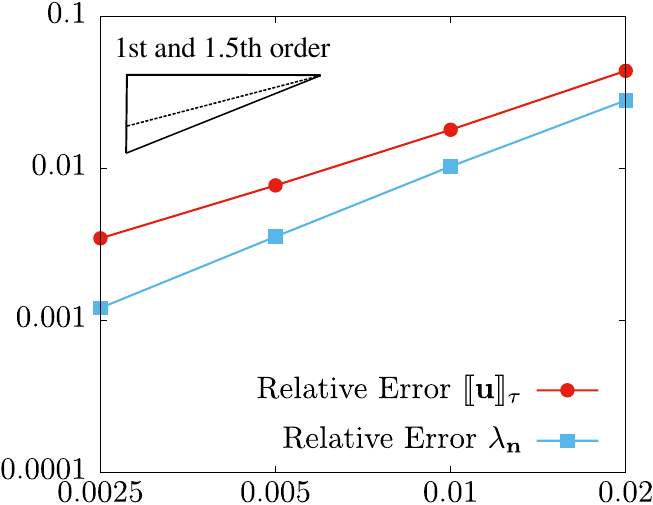}};
		\node[below=of img, node distance=0cm, rotate=0, anchor=center,yshift=0.8cm] {~~~~\footnotesize{ $h$ (m)}};
		\node[left=of img, node distance=0cm, rotate=90, anchor=center,yshift=-0.7cm] {\footnotesize{Relative $L^2$ Error}};
	\end{tikzpicture}
	\caption{Convergence of the relative $L^2$ error of $\jump{\bu}_\tang - \jump{\bu_\D}_{\D,\tang}$ and of $\lambda_\n - \lambda_{\D,\n}$ away from the tip, as a function of the size of the largest fracture face denoted by $h$. Test case of Section~\ref{tchelepi_meca}.}
	\label{compression_error}
\end{figure}

\subsection{3D manufactured solution for the Tresca friction model}\label{analytical_test}

We consider the 3D domain $\Omega = (-1,1)^3$ with the single non-immersed fracture $\Gamma =  \{0\} \times (-1,1)^2$. The Tresca coefficient $\trescacof$ is set to 1 and the Lam\'e coefficients are set to $\mu=\lambda = 1$. Note that this situation does not formally match the assumptions made in the analysis of the scheme: specifically, $\Omega\backslash\Gamma$ is not connected; however, this assumption was solely made to ensure that the Korn inequality is valid in the fractured domain, which is the case here since each connected component of $\Omega\backslash\Gamma$ sees $\partial\Omega$ (as a matter of fact, the connectedness assumption could be replaced by the assumption that each connected component of $\Omega\backslash\Gamma$ touches the boundary of $\Omega$ along some hypersurface).

The exact solution 
\begin{equation}
	\bu(x,y,z)=\label{s_term_signorini} 
	\left\{\hspace{.2cm}
	\begin{array}{lll}
		\left( {\begin{array}{c}
				h(x,y) P(z) - \trescacof y\\ 
				P(z) \\   
				x^2 P(z) \\ 
		\end{array} } \right)  & \mbox{ if } z \ge 0,\\[4ex]
		\left( {\begin{array}{c} 	
				h(x,y) Q(z) - \trescacof y\\ 
				2Q(z) \\   
				x^2 Q(z) \\ 
		\end{array} } \right)  & \mbox{ if } z < 0,~x\ < 0,  \\[1ex]
		\left( {\begin{array}{c} 	
				h(x,y) Q(z) - \trescacof y\\ 
				Q(z) \\   
				x^2 Q(z) \\ 
		\end{array} } \right)  & \mbox{ if } z < 0,~x\ \ge 0,
	\end{array}
	\right.
\end{equation}
with $h(x,y) = -\sin(x) \cos(y)$, $P(z) = z^2 $ and $Q(z) = z^2/4$, is designed to satisfy the Tresca frictional-contact conditions at the matrix fracture interface $\Gamma$.  The right hand side
$\mathbf{f} = -\div \bbsig(\bu)$ is deduced and the trace of $\bu$ is imposed as Dirichlet boundary condition on $\partial \Omega$.
Note that the fracture $\Gamma$ is in sticky-contact state for $z  > 0$ ($\jump{\bu}_\n = 0$, $\jump{\bu}_\tang = 0$) and slippy-contact for $z < 0$ ($\jump{\bu}_\n = 0$, $|\jump{\bu}_\tang| > 0$). The convergence of the mixed $\Po^1$-bubble VEM -- $\Po^0$ formulation is investigated on families of uniform Cartesian, tetrahedral, and hexahedral meshes. Starting from a uniform Cartesian mesh, an hexahedral mesh is generated by random perturbation of the nodes. This lead to non-planar faces (except on the fracture) which are dealt with either by cutting the faces into two triangles or by applying the modified gradient operator as described in remark \ref{non-planar}. These two choices illustrated in Figure \ref{rand_mesh} are denoted respectively by Hexa-cut and Hexa-bary in the following.

\begin{figure}[H]
	\centering
	\begin{tikzpicture}
		\node (img)  {\includegraphics[scale=0.3]{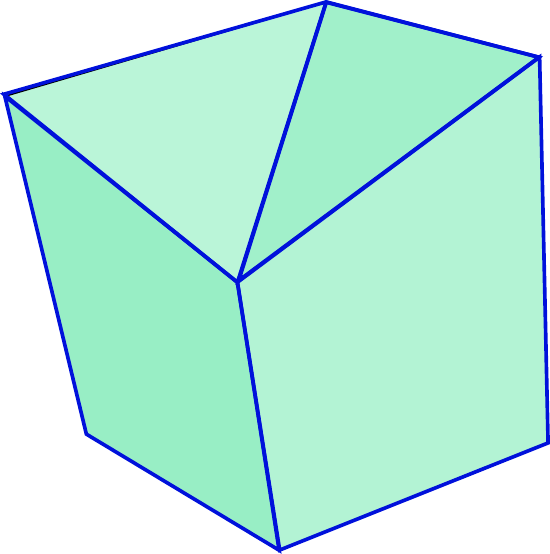}};
	\end{tikzpicture}
    \hspace{2cm}
	\begin{tikzpicture}
		\node (img)  {\includegraphics[scale=0.3]{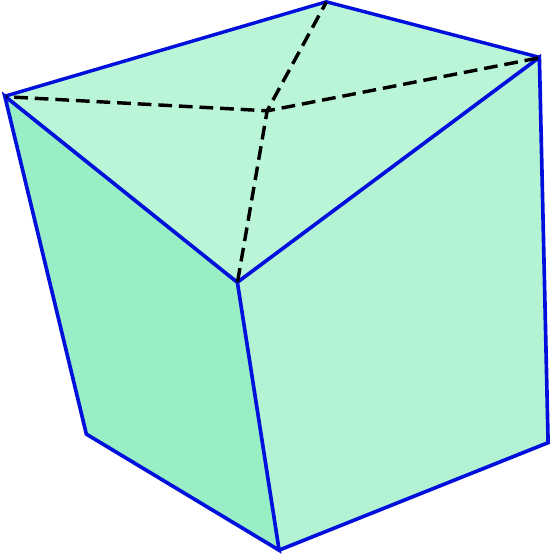}};
	\end{tikzpicture}
	\caption{Example of randomly perturbed Cartesian cell with non planar faces: hexahedral cell with a non planar face cut into two triangles (left, denoted by Hexa-cut) or four triangles using the isobarycenter of the face nodes (right, denoted by Hexa-bary). Note that in the Hexa-bary case, the displacement at the isobarycenter is eliminated by linear combination of the displacements at the face nodes as detailed in Remark \ref{non-planar}.}
	\label{rand_mesh}
\end{figure}

Figure \ref{Error_anal} exhibits the relative $L^2$ norms of the errors $\bu- \Pi_\D \bu_\D$, $\jump{\bu} - \jump{\bu_\D}_\D$, $\grad{\bu}- \nabla_\D \bu_\D$ and
$\lambda_\n - \lambda_{\D,\n}$ on the three families of refined meshes as a function of the cubic root of the number of cells.
It shows, as expected for such a smooth solution, a second-order convergence for $\bu$ and $\jump{\bu}$ with all families of meshes. A first-order convergence, coherent with Theorem \ref{error_estimate}, is obtained for $\grad{\bu}$ and $\lambda_\n$ with both the hexahedral and tetrahedral families of meshes, while a second-order convergence for $\grad{\bu}$ and a 1.5th-order convergence for $\lambda_\n$ is observed with the family of Cartesian meshes (these improved rates being probably due to the symmetry and uniformity of the mesh).
\begin{figure}[H]
	\centering
	\begin{tikzpicture}
		\node (img)  {\includegraphics[scale=0.58]{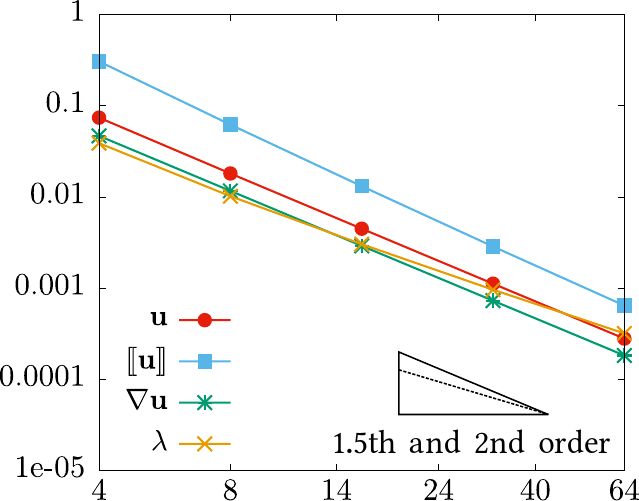}};
		\node[left=of img, node distance=0cm, rotate=90, anchor=center,yshift=-0.7cm] {\footnotesize{ $L^2$ Error }};
		\node[below=of img, node distance=0cm, yshift=1cm]  {~~~~~~~\footnotesize{ $N_{\text{cell}}^{\frac{1}{3}}$ }};
		\node[above=of img, node distance=0cm, rotate=0, anchor=center,yshift=-0.7cm]   {\footnotesize{ ~~~~(a)}};
	\end{tikzpicture}
	\begin{tikzpicture}
		\node (img)  {\includegraphics[scale=0.58]{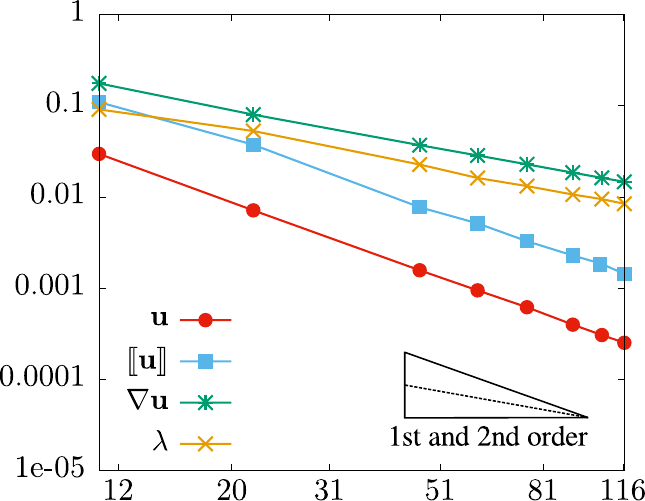}};
		\node[below=of img, node distance=0cm, yshift=1cm]  {~~~~~~~\footnotesize{ $N_{\text{cell}}^{\frac{1}{3}}$ }};
		\node[above=of img, node distance=0cm, rotate=0, anchor=center,yshift=-0.7cm]   {\footnotesize{ ~~(b)}};
	\end{tikzpicture}

    \vspace{0.2cm}
	\begin{tikzpicture}
		\node (img)  {\includegraphics[scale=0.58]{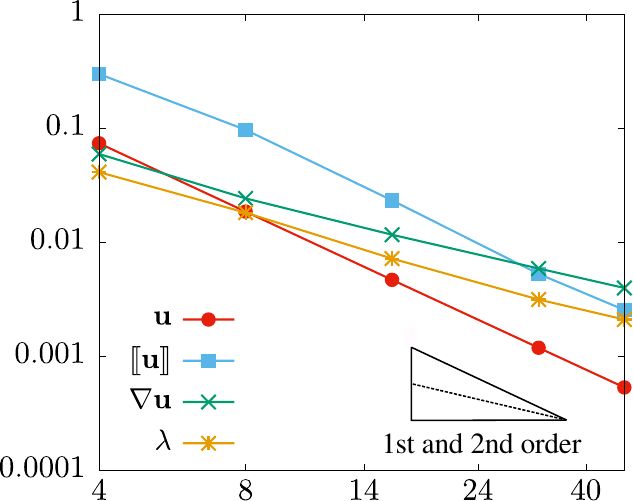}};
		 \node[left=of img, node distance=0cm, rotate=90, anchor=center,yshift=-0.7cm] {\footnotesize{ $L^2$ Error }};
		\node[below=of img, node distance=0cm, yshift=1cm]  {~~~~~~~\footnotesize{ $N_{\text{cell}}^{\frac{1}{3}}$ }};
		\node[above=of img, node distance=0cm, rotate=0, anchor=center,yshift=-0.7cm]   {\footnotesize{ ~~~~~(c)}};
	\end{tikzpicture}
	\begin{tikzpicture}
	\node (img)  {\includegraphics[scale=0.58]{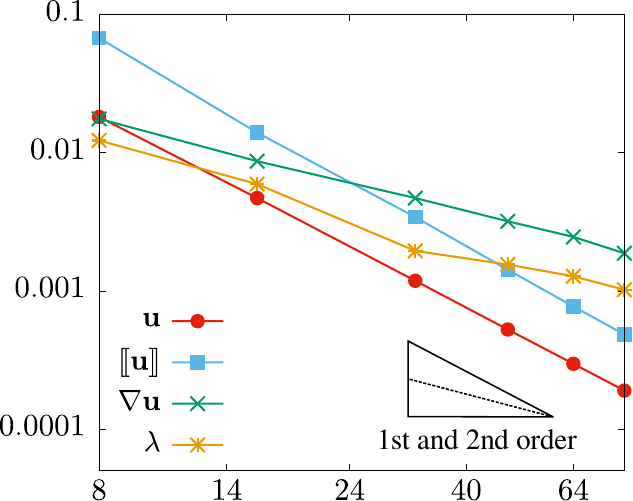}};
	\node[below=of img, node distance=0cm, yshift=1cm]  {~~~~~~~\footnotesize{ $N_{\text{cell}}^{\frac{1}{3}}$ }};
	\node[above=of img, node distance=0cm, rotate=0, anchor=center,yshift=-0.7cm]   {\footnotesize{ ~~~~~(d)}};
    \end{tikzpicture}

	\caption{Relative $L^2$ norms of the errors $\bu- \Pi_\D \bu_\D$, $\jump{\bu} - \jump{\bu_\D}_\D$, $\grad{\bu}- \nabla_\D \bu_\D$ and
		$\lambda_\n - \lambda_{\D,\n}$  as a function of the square root of the number of cells, using the families of Cartesian (a), tetrahedral (b), Hexa-cut (c) and Hexa-bary (d) meshes. Test case of Section \ref{analytical_test}.}
	\label{Error_anal}
\end{figure}
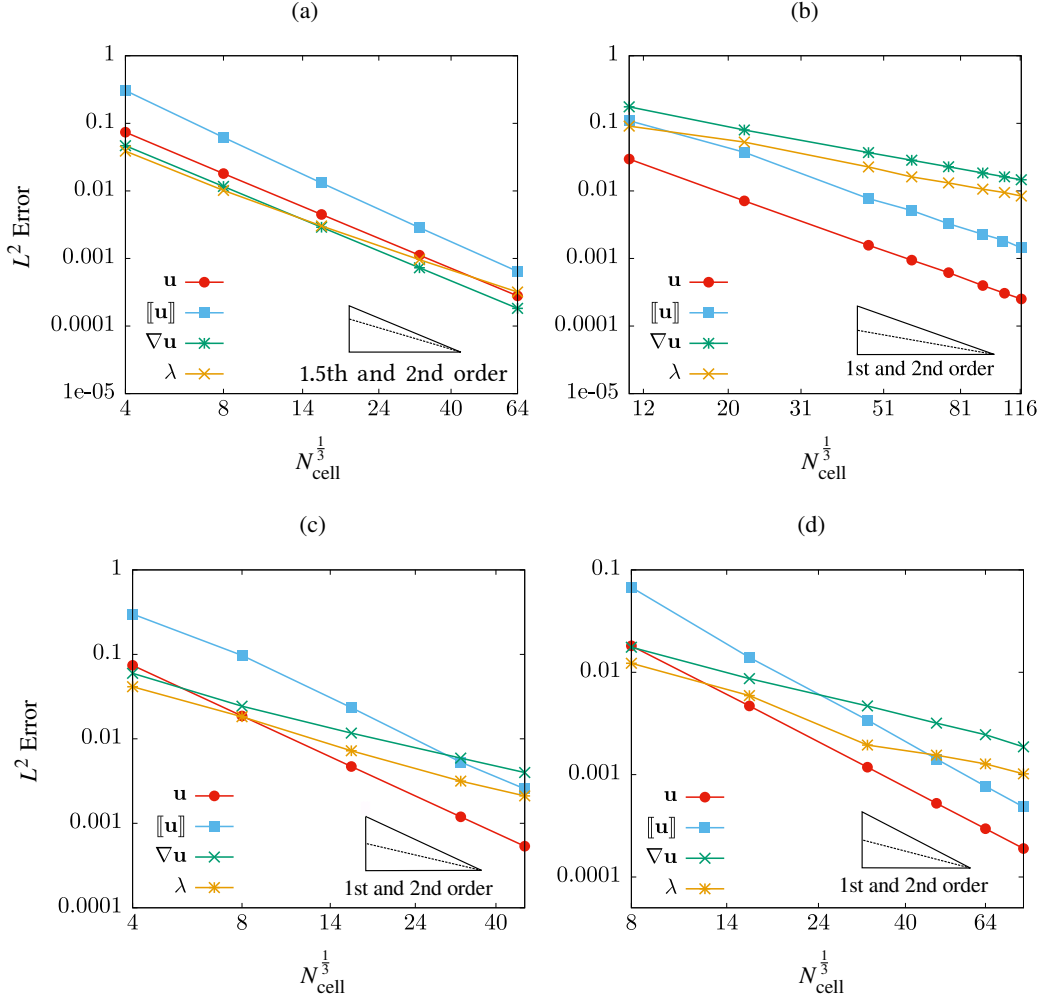

The discrete solution for the face-wise constant normal jump $\jump{\bu_\D}_{\D,\n}$ is essentially zero (machine precision of $10^{-16}$) for any mesh, as we are in a contact state on the fracture $\Gamma$. In contrast, the nodal normal jumps approach zero as the mesh is refined, as illustrated in Figure \ref{test_analytical_saut_n} for the family of Hexa-cut meshes. The nodal normal jumps in Figure \ref{test_analytical_saut_n} are plotted as a function of $z$ along the ``broken'' line corresponding to $x=y=0$ before perturbation of the mesh. Figure \ref{test_analytical_saut_t} plots, for the Hexa-cut family of meshes, the face-wise constant non-zero tangential component jump $\jump{\bu_\D}_{\D,y}$  on $\Gamma$ as well as the nodal tangential jumps as a function of $z$  as for $\jump{\bu_\D}_{\D,\n}$ in Figure \ref{test_analytical_saut_n}. We recall that, on the fracture, the exact tangential jump $\jump{\bu}_{y}$   depends only on $z$ and is equal to $\min(z/2,0)^2$.

\begin{figure}[H]
	\centering
	\begin{tikzpicture}
		\node (img)  {\includegraphics[scale=0.8]{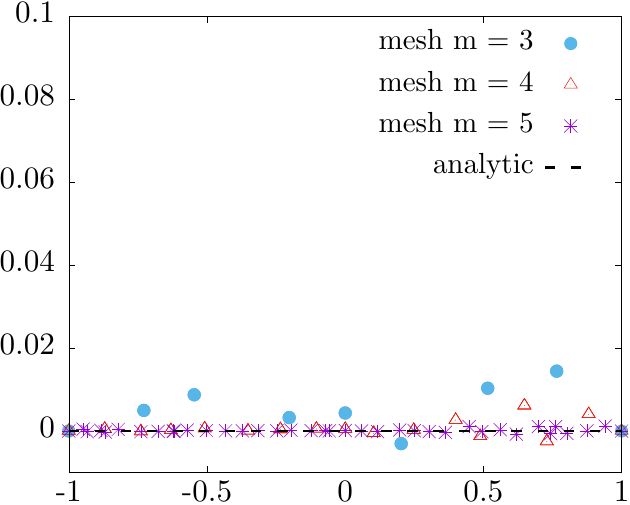}};
		\node[left=of img, node distance=0cm, rotate=90, anchor=center,yshift=-0.7cm] {\footnotesize{ $\jump{\bu }_{\mathbf{n}}$ (m)}};
		\node[below=of img, node distance=0cm, rotate=0, anchor=center,yshift=0.8cm] {~~~~\footnotesize{ $z$ (m)}};
		\node[below=of img, node distance=0cm, rotate=0, anchor=center,yshift=0.1cm] {};
	\end{tikzpicture}
	\caption{Nodal  normal jumps $\jump{\bu_\D}_{\D,\n}$ along the line $x=y=0$ as a function of $z$ for the discrete solutions on the Hexa-cut meshes with $2^{3m}$ cells, $m=3,4,5$ and for the continuous solution. Test case of Section \ref{analytical_test}.}
	
	\label{test_analytical_saut_n}
\end{figure}
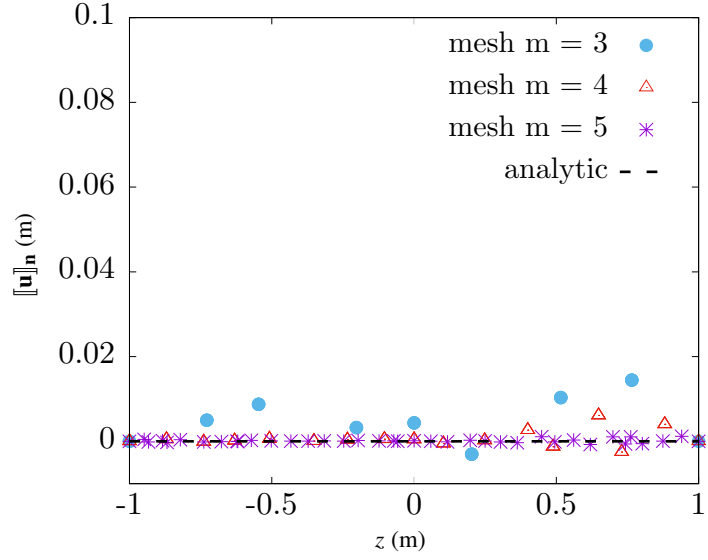

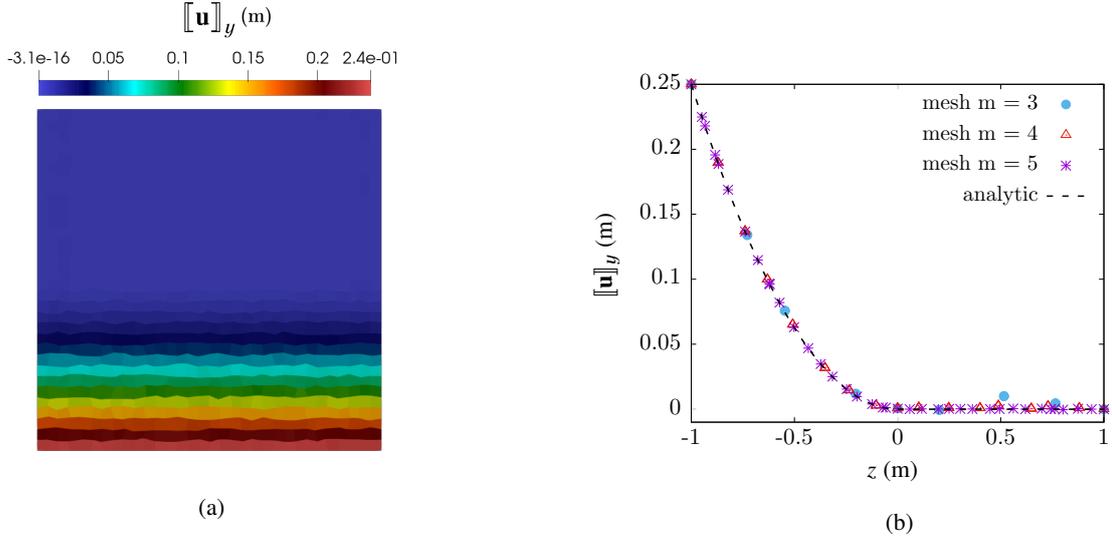
\begin{figure}[H]
	\centering
	\hspace{-3cm}
	\begin{tikzpicture}
		\raisebox{0.2cm}{\node (img)  {\includegraphics[scale=0.25]{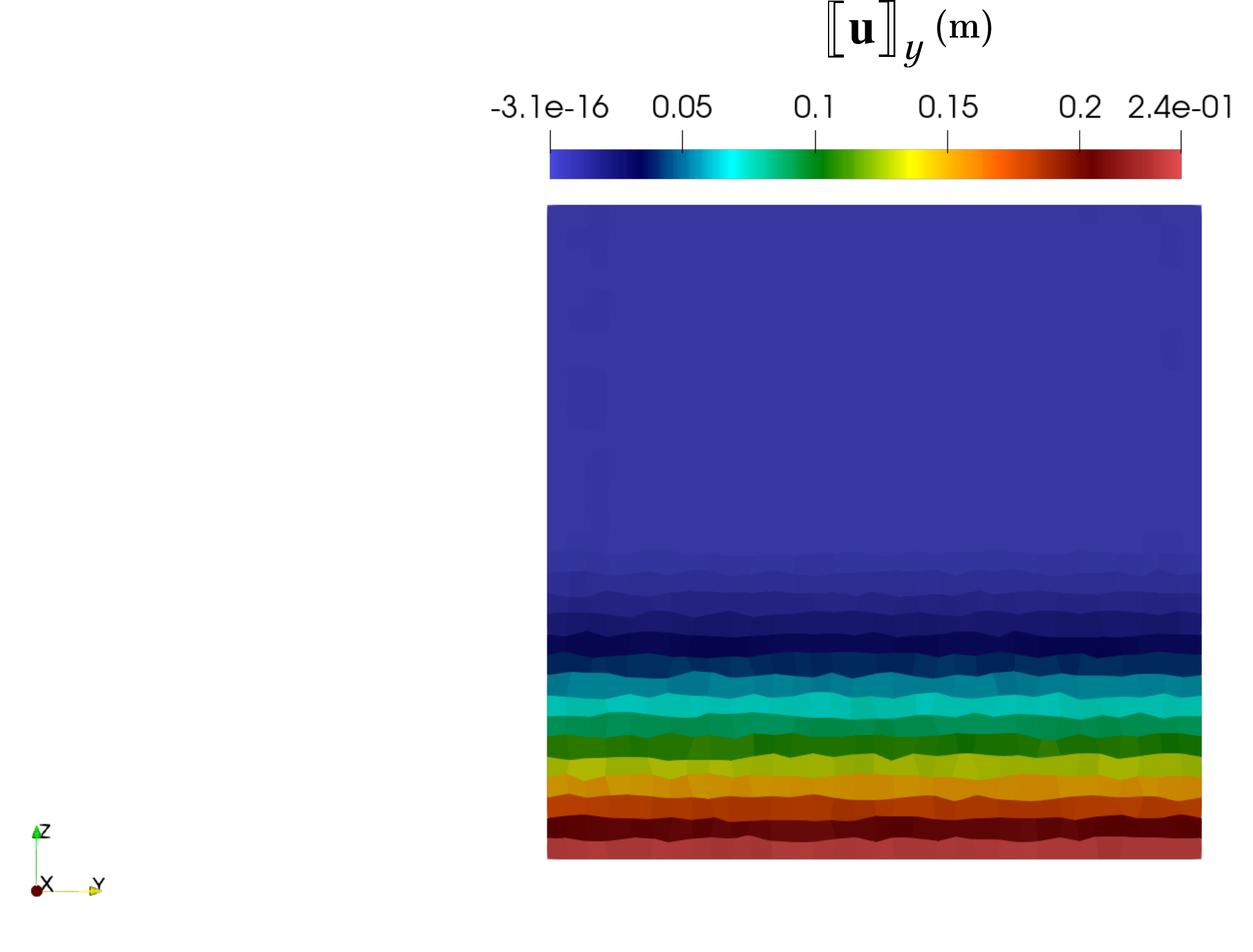}};
			\node[below=of img, node distance=0cm, rotate=0, anchor=center,yshift=1cm] {~\footnotesize{~~~~~~~~~~~~~~~~~~~~~~~~~~~~~~~~~~~~~~~~~~ (a)}};}
	\end{tikzpicture}
	\hspace{2cm}
	\begin{tikzpicture}
		\node (img)  {\includegraphics[scale=0.58]{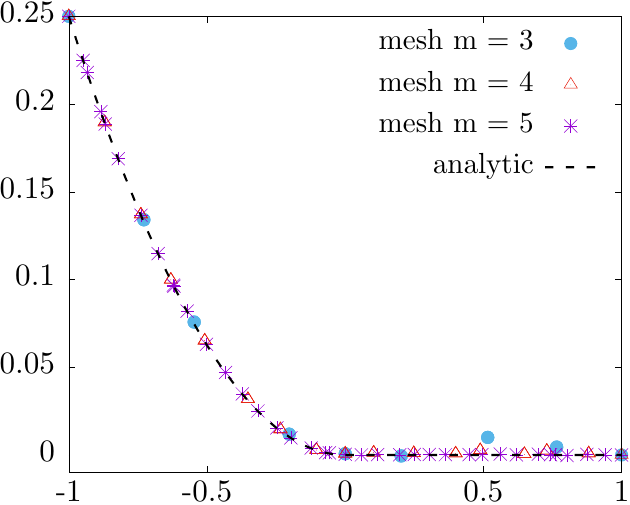}};
		\node[left=of img, node distance=0cm, rotate=90, anchor=center,yshift=-0.7cm] {\footnotesize{ $\jump{\bu }_{y}$ (m)}};
		\node[below=of img, node distance=0cm, rotate=0, anchor=center,yshift=0.8cm] {~~~~\footnotesize{ $z$ (m)}};
		\node[below=of img, node distance=0cm, rotate=0, anchor=center,yshift=0.1cm] {~~~~~~\footnotesize{ (b)}};
	\end{tikzpicture}
	\caption{(a) Face-wise constant non-zero tangential component jump $\jump{\bu_\D}_{\D,y}$ on $\Gamma$ obtained on the Hexa-cut mesh with $2^{3m}$ cells, $m=5$. (b) Nodal tangential jumps	 along the line $x=y=0$ as a function of $z$ both for the discrete solutions on the Hexa-cut meshes with $2^{3m}$ cells, $m=3,4,5$ and for the continuous solution depending only on $z$. Test case of Section \ref{analytical_test}.}
	
	\label{test_analytical_saut_t}
\end{figure}

\section{Conclusions}

We have presented in this work the convergence analysis of a fully discrete polytopal discretisation of a contact-mechanical model with Tresca friction at matrix-fracture interfaces. The analysis accounts for general elements and network of planar fractures including immersed and non-immersed fractures, with intersections, corners and tips.
The main new features are related to the polytopal bubble additional displacement degree of freedom, the proof of a discrete inf-sup condition using a specific $\JSdual$-norm to deal with networks of fractures and one-sided bubbles, and the proof of a discrete Korn inequality taking fracture networks into account. Numerical experiments based on two analytical solutions are presented and confirm the established first order error estimates on the $L^2(\Omega)$-norm of the displacement gradient  and on the $\JSdual$-norm of the Lagrange multiplier. 

\appendix

\section{Construction of $U_\Ksi$ and $\varpi_{\Ksi}$ for the averaged interpolator}\label{sec:existence.varpi}

For each $K\in\cells$ and $s\in\nodes_K$, we construct an explicit $U_\Ksi$ and $\varpi_\Ksi$ that satisfies \eqref{eq:def.varpi}. The construction shows that $U_\Ksi$ can be entirely contained in any single element in $\Ksi$. For simplicity of presentation we assume here that the space dimension is $d=3$, but the same construction can also be done in 2D.

Let $L\in \Ksi$. By mesh regularity there is a simplex $S\subset L$ that contains $s$ as one of its vertices, and that is shape-regular (with regularity factor bounded above by the global mesh regularity factor). We take $U_\Ksi=S$. We then build $\varpi_\Ksi$ on the reference simplex
$$
\widehat{S}=\mathrm{co}\{(0,0,0),(1,0,0),(0,1,0),(0,0,1)\}
$$
and linearly transport it onto $S=U_\Ksi$; the relation \eqref{eq:def.varpi.integral} are preserved by this linear transport, and the upper bound in \eqref{eq:def.varpi.bounds} is the same as the one on the reference simplex.

Without loss of generality we can therefore assume that $\mathbf{x}_s=(0,0,0)$. The simplex $\widehat{S}$ has center of mass $\mathbf{x}_{\widehat{S}}=(1/3,1/3,1/3)$, and the simplex $S_1=\frac12\widehat{S}$ has center of mass $\mathbf{x}_{S_1}=(1/6,1/6,1/6)$. We therefore have
\begin{equation}\label{eq:xs.xs1}
  \mathbf{x}_s=2\mathbf{x}_{S_1}-\mathbf{x}_{\widehat{S}}.
\end{equation}
Let us define $\varpi_\Ksi=16\mathbf{1}_{S_1}-\mathbf{1}_{\widehat{S}}$, where $\mathbf{1}_A$ is the characteristic function of $A$.
Then, since $|S_1|=|\widehat{S}|/8$,
$$
 \int_{\widehat{S}}\varpi_\Ksi = 16|S_1|-|\widehat{S}|=|\widehat{S}|,
$$
which establishes the first relation in \eqref{eq:def.varpi.integral}. To prove the second, by definition of the centers of mass we write
$$
\int_{\widehat{S}}\mathbf{x}\varpi_\Ksi=\int_{S_1} 16\mathbf{x} - \int_{\widehat{S}}\mathbf{x}
=16|S_1|\mathbf{x}_{S_1}-|\widehat{S}|\mathbf{x}_{\widehat{S}}
=|\widehat{S}|(2\mathbf{x}_{S_1}-\mathbf{x}_{\widehat{S}})=|\widehat{S}|\mathbf{x}_s,
$$
where the conclusion follows from \eqref{eq:xs.xs1}.

\section{Estimate on fracture norms}

We establish here two estimates involving the $\JSdual$-like norms, that are used in the error estimates. The first one is an approximation property of the $L^2$-projector $\ProjFG$ on piecewise constant functions.

\begin{lemma}[Approximation properties of $\ProjFG$ in $L^2$ and discrete $\JSdual$ norms]\label{lemma:approx.proj.Gamma}
We have 
\begin{equation}\label{eq:approx.Gamma.L2}
\NORM{L^2(\Gamma)}{\la-\ProjFG\la}\lesssim h_\D\SEMINORM{H^1(\faces_\Gamma)}{\la}\qquad \forall \la\in H^1(\faces_\Gamma)^d,
\end{equation}
and, for all $s\in [0,1]$,
\begin{equation}\label{eq:approx.Gamma.Hs}
\NORM{-\nicefrac12,\Gamma}{\la-\ProjFG\la}\lesssim h_\D^{\frac{1}{2}+s}\SEMINORM{H^s(\faces_\Gamma)}{\la}\qquad \forall \la\in H^{s}(\faces_\Gamma)^d.
\end{equation}
\end{lemma}

\begin{proof}
Let $\la\in H^{s}(\faces_\Gamma)^d$.
The approximation properties of $L^2$-orthogonal projectors on polynomial spaces give, for all $\sigma\in\faces_\Gamma$ and denoting by $\pi^0_\sigma$ the $L^2$-orthogonal projector on $\Poly{0}(\sigma)^d$,
$$
\NORM{L^2(\sigma)}{\la-\pi^0_\sigma\la}\lesssim h_\D^s\SEMINORM{H^s(\sigma)}{\la}
$$
(this estimate is established in \cite[Theorem 1.45]{hho-book} for the scalar case and $s\in\{0,1\}$, and can be obtained for vector-valued functions and $s\in(0,1)$ working component by component and interpolating between $s=0$ and $s=1$). 
Squaring, summing over $\sigma\in\faces_\Gamma$ and taking the square root, we infer
\begin{equation}\label{eq:la.proj.la.2}
\NORM{L^2(\Gamma)}{\la-\ProjFG\la}\lesssim h_\D^s\SEMINORM{H^s(\faces_\Gamma)}{\la}.
\end{equation}
The case $s=1$ corresponds to \eqref{eq:approx.Gamma.L2}. 

We now turn to \eqref{eq:approx.Gamma.Hs}. Given the definition \eqref{eq:def.H12.norm}, we only need to bound the norm on $\Gamma_i$ for each $i\in I$. For all $\bv_i\in H^1(\Omega_i^+;\Gamma_i)\backslash\{0\}$, by orthogonality property of $\ProjFG$ and a Cauchy--Schwarz inequality we have
\begin{align}
\int_{\Gamma_i}(\la-\ProjFG\la)\cdot\bv_i={}&\int_{\Gamma_i}(\la-\ProjFG\la)\cdot(\bv_i-\ProjFG\bv_i)\nonumber\\
\le{}& \NORM{L^2(\Gamma_i)}{\la-\ProjFG\la}\NORM{L^2(\Gamma_i)}{\bv_i-\ProjFG\bv_i}\nonumber\\
\lesssim{}& h_\D^s\SEMINORM{H^s(\faces_\Gamma)}{\la}\NORM{L^2(\Gamma_i)}{\bv_i-\ProjFG\bv_i},
\label{eq:la.proj.la.1}
\end{align}
where we have used \eqref{eq:la.proj.la.2} in the second inequality.
Letting $\ProjM$ be the $L^2$-orthogonal projector on piecewise constant functions over $\cells$ and introducing $(\ProjM\bv_i)_{|\Gamma_i}$ we have, since $\ProjFG(\ProjM\bv_i)_{|\Gamma_i}=( \ProjM\bv_i)_{|\Gamma_i}$,
\begin{align*}
\NORM{L^2(\Gamma_i)}{\bv_i-\ProjFG\bv_i}\le{}& \NORM{L^2(\Gamma_i)}{\bv_i-(\ProjM\bv_i)_{|\Gamma_i}}
	+\NORM{L^2(\Gamma_i)}{\ProjFG(\bv_i-(\ProjM\bv_i)_{|\Gamma_i})}\\
	\le{}& 2\NORM{L^2(\Gamma_i)}{\bv_i-(\ProjM\bv_i)_{|\Gamma_i}},
\end{align*}
where the conclusion comes from the boundedness in $L^2(\Gamma_i)$-norm of $\ProjFG$. The approximation properties of $\ProjM$ (see \cite[Theorem 1.45]{hho-book}) then yield
\begin{equation}\label{eq:la.proj.la.3}
\NORM{L^2(\Gamma_i)}{\bv_i-\ProjFG\bv_i}\lesssim h_\D^{\frac12}\NORM{H^1(\Omega_i^+)}{\bv_i}.
\end{equation}
Plugging \eqref{eq:la.proj.la.3} into \eqref{eq:la.proj.la.1}, dividing by $\NORM{H^1(\Omega_i^+)}{\bv_i}$ and taking the supremum over $\bv_i$ concludes the proof. \end{proof}

The second property is a bound between $H^{-1/2}$-like norms on $\Gamma$.

\begin{lemma}[Estimate of the discrete $H^{-1/2}$-norm]\label{lemma.disc.cont.norm1/2}
It holds
$$ \NORM{-\nicefrac 12,\D}{\la_\D} \lsim \NORM{-\nicefrac 12,\Gamma}{\la_\D},\quad \forall \la_{\D} \in \MD. $$
\end{lemma}

\begin{proof}
Let $\la\in\MD$. Given the definitions \eqref{eq:def.H12.norm} and \eqref{def:norm.m1/2D} of the two norms, we only need to prove that, for each $i\in I$,
there exists $\bv_i\in H^1(\Omega_i^+;\Gamma_i)^d$ such that
\begin{equation}\label{eq:equiv.norm.v}
\NORM{H^1(\Omega_i^+)}{\bv_i}\lesssim \NORM{-\nicefrac12,\D,i}{\la_\D}\quad\mbox{ and }\quad\int_{\Gamma_i}\la_\D\cdot\bv_i=\NORM{-\nicefrac12,\D,i}{\la_\D}^2,
\end{equation}
where 
$$
\NORM{-\nicefrac12,\D,i}{\la_\D}=\left(\sum_{\sigma\subset\Gamma_i}h_\sigma|\sigma| |\la_\sigma|^2\right)^{1/2}
$$
is the restriction to $\Gamma_i$ of $\NORM{-\nicefrac12,\D}{\la_\D}$. To achieve this, we first define the boundary values of $\bv_i$ and then lift these
boundary values to create $\bv_i$ itself.

\medskip

\emph{Step 1: design of bubble functions.}

For each $\sigma\subset \Gamma_i$ we take a ball $\mathcal B_\sigma\subset \sigma$ and a bubble function $b_\sigma\in W^{1,\infty}(\sigma)$ such that
\begin{equation}\label{eq:bubble.integral}
\begin{aligned}
&\mathcal B_\sigma\mbox{ has radius }\gtrsim h_\sigma\,,\quad\dist(\mathcal B_\sigma,\partial\sigma)\gtrsim h_\sigma\,,\\
&\int_\sigma b_\sigma=|\sigma|\,,\quad b_\sigma=0\mbox{ outside }\mathcal B_\sigma\,,\quad \NORM{L^\infty}{b_\sigma}\lesssim 1\,,\quad
\NORM{L^\infty}{\nabla b_\sigma}\lesssim h_\sigma^{-1}.
\end{aligned}
\end{equation}
The existence of $\mathcal B_\sigma$ is ensured by the mesh regularity assumption, and $b_\sigma$ can then be constructed on $\mathcal B_\sigma$ by scaling the function $\x\to\dist(\x,\partial\mathcal B_\sigma)$. We then have
$$
\SEMINORM{H^{1/2}(\sigma)}{b_\sigma}^2=\int_\sigma\int_\sigma\frac{|b_\sigma(\x)-b_\sigma(\y)|^2}{|\x-\y|^d}\d\x\d\y
\le \NORM{L^\infty}{\nabla b_\sigma}^2\int_\sigma\int_\sigma\frac{\d\x\d\y}{|\x-\y|^{d-2}}.
$$
Using polar coordinates around $\x$ (setting $\y=\x+\rho \boldsymbol{\xi}$ with $\rho>0$ and $\boldsymbol{\xi}$ of unit length) and recalling that $\sigma$ has dimension $d-1$ we easily get $\int_\sigma\frac{\d\y}{|\x-\y|^{d-2}}\lesssim h_\sigma$. Hence, invoking \eqref{eq:bubble.integral}, we infer
\begin{equation}\label{eq:bubble.H12}
\SEMINORM{H^{1/2}(\sigma)}{b_\sigma}^2\lesssim h_\sigma^{-1}|\sigma|.
\end{equation}

For any $\x\in\mathcal B_\sigma$, since $\x$ is at distance $\gtrsim h_\sigma$ of $\partial\sigma$ a use of polar coordinates around $\x$ yields
$\int_{\partial\Omega_i^+\backslash\sigma}\frac{\d\y}{|\x-\y|^d}\lesssim h_\sigma^{-1}$. Using \eqref{eq:bubble.integral}, we infer that
\begin{equation}\label{eq:bubble.H12.long}
\int_\sigma\int_{\partial\Omega_i^+\backslash\sigma}\frac{|b_\sigma(\x)|^2}{|\x-\y|^d}\lesssim h_\sigma^{-1}|\sigma|.
\end{equation}

\medskip

\emph{Step 2: construction of the boundary value.}

Set $\bw:\partial\Omega_i^+\to\R^d$ such that $\bw|_\sigma=h_\sigma \la_\sigma b_\sigma$ for all $\sigma\subset\Gamma_i$, and $\bw|_{\partial\Omega_i^+\backslash\Gamma_i}=0$. Then,
\begin{equation}\label{eq:w.bdry.integral}
\int_{\Gamma_i}\la_\D\cdot \bw=\sum_{\sigma\subset\Gamma_i}h_\sigma|\la_\sigma|^2\int_\sigma b_\sigma=
\sum_{\sigma\subset\Gamma_i}h_\sigma|\la_\sigma|^2|\sigma|=\NORM{-\nicefrac12,\D,i}{\la_\D}^2.
\end{equation}
By the support condition in \eqref{eq:bubble.integral}, $\bw$ is continuous (and piecewise $H^1$) on $\partial\Omega_i^+$, so it belongs to $H^{1/2}(\partial\Omega_i^+)^d$. Having in mind to create a lifting of $\bw$, we want to bound its $H^{1/2}$ norm. Specifically, we will prove that
\begin{equation}\label{eq:est.w.bdry}
\NORM{H^{1/2}(\partial\Omega_i^+)}{\bw}\lesssim \NORM{-\nicefrac12,\D,i}{\la_\D}.
\end{equation}
We only detail the bound of the seminorm, as bounding the $L^2$ norm of $\bw$ is straightforward. 
We have
\begin{align}
\SEMINORM{H^{1/2}(\partial\Omega_i^+)}{\bw}^2={}&\sum_{\sigma\subset \partial\Omega_i^+}\sum_{\sigma'\subset \partial\Omega_i^+}\int_\sigma\int_{\sigma'}\frac{|\bw(\x)-\bw(\y)|^2}{|\x-\y|^d}\d\x\d\y\nonumber\\
={}&\sum_{\sigma\subset \Gamma_i}\int_\sigma\int_\sigma\frac{|\bw(\x)-\bw(\y)|^2}{|\x-\y|^d}\d\x\d\y
+\sum_{\sigma\subset \partial\Omega_i^+}\sum_{\sigma'\subset \partial\Omega_i^+,\,\sigma'\not=\sigma}\int_\sigma\int_{\sigma'}\frac{|\bw(\x)-\bw(\y)|^2}{|\x-\y|^d}\d\x\d\y\nonumber\\
={}&\mathfrak T_1+\mathfrak T_2.
\label{eq:est.norm.w}
\end{align}
To bound $\mathfrak T_1$, we use $\bw|_\sigma=h_\sigma \la_\sigma b_\sigma$ and \eqref{eq:bubble.H12}:
\begin{equation}\label{eq:bound.T1}
\mathfrak T_1=\sum_{\sigma\subset \Gamma_i}\SEMINORM{H^{1/2}(\sigma)}{\bw|_\sigma}^2\lesssim 
\sum_{\sigma\subset \Gamma_i}h_\sigma|\sigma|\,|\la_\sigma|^2=\NORM{-\nicefrac12,\D,i}{\la_\D}^2.
\end{equation}
We now turn to $\mathfrak T_2$, for which we write
\begin{align}
\mathfrak T_2\le {}& 
2\sum_{\sigma\subset \partial\Omega_i^+}\sum_{\sigma'\subset \partial\Omega_i^+,\,\sigma'\not=\sigma}\int_\sigma\int_{\sigma'}\frac{|\bw(\x)|^2}{|\x-\y|^d}\d\x\d\y
+
2\sum_{\sigma\subset \partial\Omega_i^+}\sum_{\sigma'\subset \partial\Omega_i^+,\,\sigma'\not=\sigma}\int_\sigma\int_{\sigma'}\frac{|\bw(\y)|^2}{|\x-\y|^d}\d\x\d\y\nonumber\\
={}&4\sum_{\sigma\subset \partial\Omega_i^+}\int_\sigma\int_{\partial\Omega_i^+\backslash\sigma}\frac{|\bw(\x)|^2}{|\x-\y|^d}\d\x\d\y\nonumber\\
\lesssim{}&\sum_{\sigma\subset \partial\Omega_i^+}h_\sigma|\sigma|\,|\la_\sigma|^2=\NORM{-\nicefrac12,\D,i}{\la_\D}^2,
\label{eq:bound.T2}
\end{align}
where the second line follows gathering the two terms in the right-hand side of the first line by symmetric roles of $\sigma$ and $\sigma'$, and
the last line from $\bw|_\sigma=h_\sigma \la_\sigma b_\sigma$ and \eqref{eq:bubble.H12.long}. 

Plugging \eqref{eq:bound.T1} and \eqref{eq:bound.T2} into \eqref{eq:est.norm.w} concludes the proof of \eqref{eq:est.w.bdry}.

\medskip

\emph{Step 3: conclusion.}

Since $\bw \in H^{1/2}(\partial\Omega_i^+)^d$ we can find a lifting $\bv_i\in H^1(\Omega_i^+)^d$ of $\bw$ such that $\NORM{H^1(\Omega_i^+)}{\bv_i}\lesssim \NORM{H^{1/2}(\partial\Omega_i^+)}{\bw}$. Since $\bw$ vanishes outside $\Gamma_i$, we have $\bv_i\in H^1(\Omega_i^+;\Gamma_i)$. By \eqref{eq:est.w.bdry} and \eqref{eq:w.bdry.integral}, the relations \eqref{eq:equiv.norm.v} hold and the proof is complete. \end{proof}

\section*{Acknowledgement}
We acknowledge the partial support from the joint laboratory IFPEN-Inria Convergence HPC/AI/HPDA for the energetic transition in realizing this project. We would like to thank Isabelle Faille and Guillaume Ench\'ery from IFPEN for fruitful discussions and comments during the elaboration of this work. 

This research pas partially funded by the European Union (ERC Synergy, NEMESIS, project number 101115663). Views and
opinions expressed are however those of the authors only and do not necessarily reflect those of the
European Union or the European Research Council Executive Agency. Neither the European Union nor
the granting authority can be held responsible for them.

%
\bibliographystyle{plain}
\bibliography{PolytopalBubbleAnalysis.bib}
\end{document}